\documentclass[11pt,leqno,a4paper,english]{amsart}
\usepackage{amssymb,amsmath,amsfonts,amsthm,amscd}
\usepackage{leftidx}
\usepackage{tikz}
\usetikzlibrary{matrix,positioning,decorations.pathreplacing}

\usepackage[utf8]{inputenc}

\usepackage{mathrsfs}

\usepackage{hyperref}
\hypersetup{
    colorlinks=true,       
    linkcolor=blue,          
    citecolor=blue,        
    filecolor=magenta,      
    urlcolor=cyan           
}

\usepackage{subfigure,graphicx}
\usepackage{xcolor}
\usepackage{xspace}

\usepackage{enumerate}

\usepackage[all]{hypcap}

\newcommand{\mb}[1]{\ensuremath{\mathbb{#1}}}
\newcommand{\N}{{\mb{N}}}
\newcommand{\ZZ}{{\mb{Z}}}
\newcommand{\R}{{\mb{R}}}
\newcommand{\C}{{\mb{C}}}

\newcommand{\eps}{\varepsilon}

\newcommand{\unitfunction}[1]{\bld{1}_{#1}}

\DeclareMathOperator{\Id}{Id} 

\newcommand{\D}{\ensuremath{\mathscr D}}
\newcommand{\E}{\ensuremath{\mathcal E}}

\renewcommand{\H}{\ensuremath{\mathcal H}}

\newcommand{\M}{\ensuremath{\mathcal M}}

\renewcommand{\k}{\ensuremath{\kappa}}

\newcommand{\T}{\mathsf T}

\DeclareMathOperator{\Span}{span}

\newcommand{\imp}{\Rightarrow}

\newcommand{\n}{\mathsf n}

\newcommand{\Con}{\ensuremath{\mathscr C}}

\newcommand{\Cinfc}{\ensuremath{\mathscr C^\infty_{c}}}

\renewcommand{\d}{\ensuremath{\partial}}

\newcommand{\rhs}{r.h.s.\@\xspace}
\newcommand{\lhs}{l.h.s.\@\xspace}

\newcommand{\eg}{e.g.\@\xspace}
\newcommand{\resp}{resp.\@\xspace}

\newcommand{\nhd}{neighborhood\xspace}

\newcommand{\suff}{sufficiently\xspace}

\DeclareMathOperator{\sgn}{sgn}

\DeclareMathOperator{\obs}{obs}

\DeclareMathOperator{\Obs}{\mathsf L}

\newcommand{\bld}[1]{\mbox{\boldmath $#1$}}

\newcommand{\tF}{\tilde{F}}
\newcommand{\tJ}{\tilde{J}}

\let \div \relax
\DeclareMathOperator{\div}{div}
\newcommand{\nablag}{\nabla_{\!\! g}}
\newcommand{\divg}{\div_{\! g}}

\newcommand{\hk}{h_k}

\newcommand{\Aell}{\mathsf A}

\DeclareMathOperator{\dist}{dist}

\DeclareMathOperator{\supp}{supp}

\newcommand{\inp}[2]{(#1, #2)} 
\newcommand{\biginp}[2]{\big(#1, #2 \big)}

\newcommand{\ovl}[1]{\overline{#1}}
\newcommand{\udl}[1]{\underline{#1}}

\newtheoremstyle{note}{} {}{\itshape}{-6pt}{\bf}{. --}{ }{}

\newtheorem{theorem}{Theorem}[section]

\newtheorem{lemma}[theorem]{Lemma}
\newtheorem{corollary}[theorem]{Corollary}
\newtheorem{assumption}[theorem]{Assumption}

\newtheorem*{theo*}{Theorem}

\newtheorem*{scproperty}{Semi-classical observability property}

\newcounter{theorembiss}

\newcounter{lemmabiss}

\theoremstyle{definition}

\newtheorem{remark}[theorem]{Remark}

\newtheorem{example}[theorem]{Example}

\DeclareMathOperator{\loc}{loc}

\renewcommand{\S}{\ensuremath{\mathscr S}}

\newcommand{\Norm}[2]{{\| #1 \|}_{#2}}
\newcommand{\bigNorm}[2]{{\big\| #1 \big\|}_{#2}}

\newcommand{\bt}{_{|t=0}}

\newcommand{\indj}{j}

\newcommand{\LS}{Lopatinski\u\i-\v{S}apiro\xspace}

\numberwithin{equation}{section}

\title[semi-classical observation] 
{Semi-classical observation suffices for observability: wave and
  Schr\"odinger equations} 

\author{Nicolas Burq}
\address{Nicolas Burq. Laboratoire de Math\'ematiques d'Orsay, Universit\'e
  Paris-Sud, Universit\'e Paris-Saclay, B\^atiment~307, 91405
  Orsay Cedex \& CNRS UMR 8628 \& Institut Universitaire de France}
 \email{nicolas.burq@math.u-psud.fr}
\author{Belhassen Dehman}
  \address{Belhassen Dehman. Universit\'e de Tunis El Manar, Facult\'e
  des Sciences de Tunis, 2092 El  Manar \& Ecole Nationale d'Ing\'enieurs de Tunis, ENIT-LAMSIN, B.P. 37, 1002 Tunis, Tunisia. }
\email{Belhassen.Dehman@fst.utm.tn}
 \author{J\'er\^ome Le Rousseau}
 \address{J\'er\^ome Le Rousseau.
   Universit\'e Sorbonne Paris Nord, Laboratoire Analyse, G\'eom\'etrie et Applications, LAGA, CNRS, UMR 7539, F-93430, Villetaneuse, France.}
\email{jlr@math.univ-paris13.fr}
\date{\today}

\begin{document}
\begin{abstract}
  For the wave and the Schr\"odinger equations we show how
  observability can be deduced from the
  observability of solutions localized in  frequency according to a 
  dyadic  scale.
\end{abstract}

\maketitle
\tableofcontents
\section{Waves and observability}

On a bounded smooth open set $\Omega$ of $\R^d$, consider the operator $A = - \Delta
= - \sum_{1\leq j\leq d} \d_j^2$. The associated wave equation in the
case of homogeneous Dirichlet boundary conditions is
\begin{align}
  \label{eq: wave equation-observation-intro}
  \begin{cases}
     (\d_t^2 + A)  \, u =0
     & \text{in}\ \R\times\Omega,\\
     u=0 &  \text{in}\ \R \times\d\Omega,\\
      u_{|t=0}  = \udl{u}^0, \ \d_t u_{|t=0}  = \udl{u}^1 
      & \text{in}\ \Omega.
  \end{cases}
\end{align}

\subsection{Strong and weak solutions}
For $\udl{u}^0 \in H^2(\Omega) \cap H^1_0(\Omega)$ and $\udl{u}^1\in H^1_0(\Omega)$,
there exists a unique 
\begin{align*}
    &u \in \Con^0 \big(\R;  H^2(\Omega) \cap H^1_0(\Omega) \big)
    \cap \Con^1 \big(\R;  H^1_0(\Omega)\big),
\end{align*}
solution to \eqref{eq: wave
  equation-observation-intro}. Such a solution is called a strong
solution as $ (\d_t^2 + A)  \, u =0$ holds in $L_{\loc}^2\big(\R; L^2(\Omega)\big)$.
One denotes by 
\begin{align*}
  \E_{2} (u) (t)
  &= \frac12 \big(  \Norm{A u(t)}{L^2(\Omega)}^2
    + \Norm{\d_tu(t)}{H^1_0(\Omega)}^2 \big)
\end{align*}
the (strong) energy of $u$ at time $t$. 
Since the equation \eqref{eq: wave
    equation-observation-intro} is
homogeneous this energy
is independent of time $t$ that is, 
\begin{align*}
  \E_{2} (u) (t) = \E_{2} (u) (0) =  \frac12 \big(
  \Norm{A \udl{u}^0}{L^2 (\Omega)}^2
  + \Norm{\udl{u}^1}{H^1_0(\Omega)}^2 \big).
\end{align*}
One  thus simply writes $\E_{2} (u)$. In particular this
conservation of the energy states the continuity of the map
\begin{align}
  \label{eq: continuity wave initial conditions}
  \big( H^2(\Omega) \cap H^1_0(\Omega) \big) \oplus H^1_0(\Omega)
  &\to \Con^0\big(\R; H^2(\Omega) \cap H^1_0(\Omega)\big)  \cap
                             \Con^1\big(\R; H^1_0(\Omega)\big)\\
  (\udl{u}^0, \udl{u}^1) & \mapsto u \notag.
\end{align}

For less regular intitial data one uses a notion of weak solution.
For instance, if $\udl{u}^0 \in H^1_0(\Omega)$ and $\udl{u}^1\in
L^2(\Omega)$, there exists a unique 
\begin{align*}
    &u \in \Con^0 \big(\R; H_0^1(\Omega)  \big)
    \cap \Con^1 \big(\R; L^2(\Omega)\big),
  \end{align*}
  that is a weak solution of \eqref{eq: wave
    equation-observation-intro}, meaning $ u\bt = \udl{u}^0$ and
  $\d_t u\bt = \udl{u}^1$ and $(\d_t^2 + A) u =0$ holds in
  $\D' \big(\R \times \Omega\big)$.
For such a solution one considers the following energy
\begin{align*}
  \E_{1} (u) (t)
  &= \frac12 \big(  \Norm{u(t)}{H^1_0(\Omega)}^2
    + \Norm{\d_tu(t)}{L^2(\Omega)}^2 \big)
\end{align*}
independent of time $t$ as above, that is, 
\begin{align*}
  \E_{1} (u) = \E_{1} (u) (t) = \E_{1} (u) (0) =  \frac12 \big(
  \Norm{\udl{u}^0}{H^1_0(\Omega)}^2
  + \Norm{\udl{u}^1}{L^2(\Omega)}^2 \big).
\end{align*}
With the density of $H^2(\Omega) \cap H^1_0(\Omega)$ in $H^1_0(\Omega)$ and of $H^1_0(\Omega)$
in $L^2(\Omega)$, one can approach $\udl{u}^0$ and $\udl{u}^1$ by
smoother data, and thus approach $u(t)$ by strong solutions.

\subsection{Observation operator, admissiblity and observability}
An observation operator is an operator $\Obs$ on $L^2(\Omega)$,
possibly unbounded,  with
values in a Hilbert space $K$.
Basis examples in the framework of the present introduction are the
following ones. 
\begin{example}
  \label{ex: observation operator}$\phantom{-}$
  \begin{enumerate}
  \item If $\omega$ is an open subset of $\Omega$ one can define
    $\Obs: v \mapsto \unitfunction{\omega} v$, yielding a bounded
    operator on $L^2(\Omega)$. 
  \item \label{ex: observation operator: neumann trace} If $\Gamma$ is an open set of $\d\Omega$ one can define
    $\Obs: v \mapsto \unitfunction{\Gamma} \d_\n v_{|\d\Omega}$,
    where $\n$ is the outgoing normal vector at $\d\M$,  yielding an unbounded
    operator on $L^2(\Omega)$. 
    \end{enumerate}
\end{example}

The observation operator  is said to satisfy an admissibility
condition if an estimate of the following form holds
\begin{align*}
  \int_0^S \Norm{\Obs u(t)}{K}^2 d t \leq C \E_j(u),
\end{align*}
for some $S>0$, $C>0$ and an energy level $j=1$ or $2$ (other energy
levels are considered in the abstract development in what follows).

For example, let us assume here that $j=2$, that is, admissibility is
given at the level of strong solutions. 
One says that observability holds with the operator $\Obs$ in time $T>0$ if
one has 
\begin{align*}
  \E_\ell(u) \leq C_{\obs}\int_0^T \Norm{\Obs u(t)}{K}^2 d t, 
\end{align*}
with $\ell=1$ or $2$, for some $C_{\obs}>0$ for any strong solution to
wave equation. If $\ell=1$ one says that observability holds with some loss of derivative, 
or some loss of energy, here, a loss of one energy level. 

\medskip
Observability estimates are important in applications such as inverse
problems or controllability issues. In particular, for waves,
observability is equivalent to exact controllability; see \eg
\cite{Coron:2007}. For more aspects on  admissibility, 
observability and their connections with controllability, we refer the reader to the book of M.~Tucsnak et G.~Weiss \cite{TW:09}.

\subsection{Derivation of an observability estimate}

There are various methods to derive observability estimates for the
wave equation. Some rely on a multiplier approach going back to the
seminal work of J.-L.~Lions \cite{Lions:88}. Others rely on microlocal
methods following the celebrated article of C.~Bardos, G.~Lebeau, and
J.~Rauch \cite{BLR:92}. 

The purpose of the present article is not the derivation of
observability {\em per se}. We are rather interested in showing that
observability, be it with energy loss or not, can be deduced from
the observation of very particular types of waves. The waves we
shall consider are localized in a frequency band making them easier to
handle than general waves (in particular when applying microlocal techniques). The frequency band is indexed by an integer
$k$ and ranges from $\alpha \rho^{|k|}$ to $\rho^{|k|}/\alpha$ for
$0< \alpha< 1$ and some $\rho>1$. This framework is given a
semi-classical aspect by using the small parameter
$\hk = \rho^{-|k|}$.

If $u^k$ denotes a wave localized in frequency as described above, a
very pleasant property is that $u^k$ fufills the half-wave equation
\begin{align}
  \label{eq: half wave - intro}
  (\d_t - \sgn(k) i A^{1/2}) u = 0.
\end{align}
This can greatly simplify the analysis necessary for the derivation of
an observation inequality as compared to treating all 
solutions to the wave equation. Also, the frequency  localization of
$u^k$ allows one to use powerful tools from semi-classical analysis
that are often easier to handle that the analogous tools from
microlocal analysis. The use of such tools can allow one  to treat the case of coefficients with
limited regularity; see for instance \cite{BDLR1} for this last
point. Having in mind the analysis of the HUM control operator carried
out in \cite{DL:09} the introduction of waves with frequencies limited
to a narrow band is very natural. In  \cite{DL:09}, the authors
show that the control operator acts
microlocally with a highly separated treatment of
frequency bands similar to those considered here.

\medskip
The starting point of the present article is to assume that
a uniform observability estimate holds for frequency localized waves like $u^k(t)$, that
is, for some $C_{\obs}>0$ one has 
\begin{align}
  \label{eq: semicl observability inequality-intro}
  \E_{\ell} (u^k) \leq C_{\obs}\int_0^T \Norm{\Obs u^k(t)}{K}^2\,  d
  t, 
\end{align}
for all $k$ \suff large. 
Our main result, under a unique continuation property to be described below, is the derivation from
\eqref{eq: semicl observability inequality-intro} of
the observability inequality for general waves $u(t)$ in the considered
energy level
\begin{align*}
  \E_{\ell} (u) \leq C_{\obs}' \int_0^{T'} \Norm{\Obs u(t)}{K}^2\,  d t,
\end{align*}
for any $T' >T$ and some $C_{\obs}'>0$.
We shall also show that an admissibility condition can be used to give
the proper energy level where this inequality holds. 

\medskip
To allow for a general use of this result, we present it in a general
abstract framework.

\subsection{Schr\"odinger equation}

In the same geometrical setting as above, the Schr\"odinger equation,
in the case of Dirichlet boundary conditions reads
\begin{align}
  \label{eq: schrodinger equation-observation-intro}
  \begin{cases}
     (i \d_t  + A)  \, u =0
     & \text{in}\ \R\times\Omega,\\
     u=0 &  \text{in}\ \R \times\d\Omega,\\
      u_{|t=0}  = \udl{u}^0
      & \text{in}\ \Omega.
  \end{cases}
\end{align}
For $\udl{u}^0 \in D(A) = H^2(\Omega) \cap H^1_0(\Omega)$ there exists a
unique solution in 
\begin{align*}
    &u \in \Con^0 \big(\R;  H^2(\Omega) \cap H^1_0(\Omega) \big)
    \cap \Con^1 \big(\R;  L^2(\Omega)\big),
\end{align*}
solution to \eqref{eq: wave
  equation-observation-intro} and $ (i \d_t  + A)  \, u =0$ holds in $L_{\loc}^2\big(\R; L^2(\Omega)\big)$.
In fact, the norm
\begin{align*}
  \Norm{u(t)}{D(A)} = \Norm{A u(t)}{L^2(\Omega)}, 
\end{align*}
is independant of $t$. 
As for the wave equation, other levels of regularity are possible. If
$\udl{u}^0 \in H^1_0(\Omega)$ there exists a
unique 
solution in 
\begin{align*}
    &u \in \Con^0 \big(\R;  H^1_0(\Omega) \big)
    \cap \Con^1 \big(\R;  H^{-1}(\Omega)\big),
\end{align*}
and the norm $\Norm{u(t)}{H^1_0(\Omega)}$ remains constant.
If $\udl{u}^0 \in L^2(\Omega)$ there exists a
unique 
solution in 
\begin{align*}
    &u \in \Con^0 \big(\R;  L^2(\Omega) \big)
    \cap \Con^1 \big(\R;  D(A)'\big),
\end{align*}
and the norm $\Norm{u(t)}{L^2(\Omega)}$ remains constant.

For an observation opertator as above, observability takes the form 
\begin{align*}
  \Norm{\udl{u}^0}{D(A^\ell)} \leq C_{\obs} \int_0^T \Norm{\Obs
  u}{K}^2 \, d t,
\end{align*}
here at the regularity given by $D(A^\ell)$, for $\ell=0, 1/2$
or $1$ in the above levels of solutions. 
As for the wave equation, under a unique continuation property,  we shall derive such
an observability inequality from a similar inequality holding for
solutions  localized in frequency.

\medskip The Schr\"odinger equation can be seen sometimes as a half-wave
equation; compare \eqref{eq: schrodinger equation-observation-intro}
and \eqref{eq: half wave - intro}.  With respect to the analysis we
carry out in the present paper, this comparison is very relevant and
the analysis is more involved for the wave equation.  In what follows,
we shall thus cover the wave equation first and cover the case of the
Schr\"odinger on a second pass, yet with all necessary details. 

\subsection{Other settings}
In this introductory section we have concentrated our attention on the
case of the wave and the Schr\"odinger equations stated on a bounded
smooth open set $\Omega$ of $\R^d$, along with homogeneous Dirichlet
boundary conditions, that is, $B u = 0$ with $B u= u_{|\d\Omega}$. 
This is done for the purpose of motivation. However, the abstract
framework we present in what follows allows one to consider more
general settings. We give a nonexhaustive list of such settings.
\begin{enumerate}
\item One can consider the elliptic operator $A$ to be the
  Lapace-Beltrami (up to principal part with the requirement that $A$
  be selfadjoint and nonnegative) on a smooth Riemannian manifold $\M$
  without boundary.  If viewed as an unbounded operator on $L^2(\M)$,
  one sees that $0$ is an eigenvalue associated
  with constant functions. Considering the operator acting on $L^2(\M) / \C$
  one then obtains the setting developped in what follows.
  \item On a  bounded smooth open set or on a smooth Riemannian
    manifold $\M$ with boundary, one can consider Neumann boundary
    conditions, that is, $B u = 0$ with $B u= \d_\n u_{|\d\M}$,
    with $\n$ the outgoing normal vector at $\d\M$.   The operator $A$ can be the Laplace(-Beltrami) 
    operator.  Similarly to the case without boundary, $0$
    is an eigenvalue of the elliptic operator $A$ associated with
    constant functions. The same quotient procedure yields a setting
    compatible with the analysis developped in what follows.  More
    generally, one can consider a boundary operator
    $B$ that fulfills the more general \LS
    boundary condition that encompasses both Dirichlet and Neumann
    conditions, with the requirement that the considered elliptic
    operator be selfadjoint and nonnegative; we refer for instance to \cite[Chapters 2
    and 4]{LRLR:V2}. Then, one has to consider a quotient with respect to 
    the kernel of the resulting unbounded operator if this kernel is
    not trivial.
  \item Above, the coefficients of the elliptic operator are considered
    smooth. This can be relaxed, down to Lipschitz regularity, yet
    preserving the properties needed in what follows. Similarly, the
    regularity of the open set $\Omega$ or the manifold $\M$ (and its
    boundary $\d\M$) can be
    chosen as low as $W^{2,\infty}$.
\end{enumerate}

\section{Abstract equations and semi-classical reduction}

Let $E$ be a Hilbert space. Consider a positive unbounded selfadjoint
operator $\Aell$ on $E$ with dense domain $D(\Aell)$.  Assume that
there exists a real Hibert basis $(e_\nu)_{\nu \in \N}$ of $E$, associated
with a nondecreasing sequence of eigenvalues,
$(\lambda_\nu)_{\nu \in \N}$, with $\lambda_\nu \to +\infty$ as
$\nu \to +\infty$, for instance if $\Aell$ has a compact resolvent
map.  In the example of the introduction, one has $E= L^2(\Omega)$ and
$\Aell = - \Delta$ with $D(\Aell) = H^2(\Omega) \cap H^1_0(\Omega)$.

Any $u \in E$ reads $u = \sum_{\nu \in \N} u_\nu \, e_\nu$ with $u_\nu  = \inp{u}{e_\nu}_E$ and $(u_\nu)_\nu \in \ell^2(\C)$.  

For $s\geq 0$ one has 
\begin{align*}
  D(\Aell^s) = \{ u\in E; \ (\lambda_\nu^s u_\nu)\in \ell^2(\C)\}.
\end{align*}
For $s <0$, $D(\Aell^{s})$ denotes the dual of $D(\Aell^{|s|})$ using $E$ as a
pivot space, and if $u\in D(\Aell^{s})$ then $u = \sum_{\nu \in \N} u_\nu e_\nu$ with
convergence for  the natural dual norm on $D(\Aell^{s})$ and $\Aell^s
u = \sum_{\nu \in \N} \lambda_\nu^s u_\nu e_\nu \in E$.
In all cases, a norm on $D(\Aell^{s})$ is given by 
\begin{align*}
  \Norm{u}{D(\Aell^{s})}^2 = \Norm{\Aell^{s} u}{E}^2 
  = \bigNorm{ (\lambda_\nu^s u_\nu)_\nu}{\ell^2(\C)}^2
  = \sum_{\nu \in \N} \lambda_\nu^{2s} | u_\nu|^2,
\end{align*}
with the associated innerproduct
$\inp{u}{v}_{D(\Aell^{s})} 
= \inp{\Aell^{s} u}{\Aell^{s} v}_{E}= \biginp{(\lambda_\nu^s u_\nu)_\nu}{(\lambda_\nu^s
  v_\nu)_\nu}_{\ell^2(\C)}$.
One has the continuous and dense injection $D(\Aell^s) \hookrightarrow
D(\Aell^{s'})$ if $s \geq s'$, moreover compact if $s >s'$.
In fact, one defines $D(\Aell^\infty) = \bigcap_{s \in \R} D(\Aell^s)$. If $u = 
\sum_{\nu \in \N} u_\nu e_\nu \in D(\Aell^s)$, one sees that $U_n =
\sum_{\lambda_\nu \leq n} u_\nu e_\nu \in D(\Aell^\infty)$ and $U_n
\to u$ in $D(\Aell^s)$ as $n \to \infty$. Hence the injection
$D(\Aell^\infty)  \hookrightarrow D(\Aell^s)$ is dense for any $s \in
\R$.

\subsection{Abstract wave equation and energy levels}
\label{sec: Abstract wave equation and energy levels}
The wave equation reads
\begin{align}
  \label{eq: abstract wave equation}
  \d_t^2 u + \Aell u = 0, \qquad u\bt=\udl{u}^0, \ \d_t u\bt=\udl{u}^1.
\end{align}
With the initial conditions $\udl{u}^0 \in E = D(\Aell^0)$ and $\udl{u}^1\in D(\Aell^{-1/2})$, the unique solution to \eqref{eq: abstract wave equation}
in $\Con^0(\R; E)\cap \Con^1\big(\R; D(\Aell^{-1/2})\big)$
is given by 
\begin{align}
  \label{eq: semicl reduc wave eq solution}
  u(t) = \sum_{\nu \in \N} \Big(\cos(t \sqrt{\lambda_\nu}) \udl{u}^0_\nu
  + \frac{1}{\sqrt{\lambda_\nu}}\sin(t \sqrt{\lambda_\nu})\udl{u}^1_\nu\Big) e_\nu
  = \sum_{\nu \in \N}\big( e^{i t \sqrt{\lambda_\nu}} u_\nu^+
    + e^{-i t \sqrt{\lambda_\nu}} u_\nu^-\big) e_\nu,
\end{align}
with 
$u_\nu^\pm =(\udl{u}^0_\nu\mp i  \udl{u}^1_\nu /
\sqrt{\lambda_\nu})/2$. 
Moreover, one has $u \in \cap_k \Con^k(\R; D(\Aell^{-k/2})$.
Note
that $(u_\nu^\pm)_{\nu\in \N} \in \ell^2(\C)$.
In turn the \rhs of
\eqref{eq: semicl reduc wave eq solution} is solution to the wave
equation~\eqref{eq: abstract wave equation}
with $\udl{u}^0$ and $\udl{u}^1$ given by
\begin{align}
  \label{eq: semicl reduc wave coeff}
  \udl{u}^0_\nu=  u_\nu^+ +  u_\nu^-
  \ \ \text{and} \ \ 
 \udl{u}^1_\nu =  i \sqrt{\lambda_\nu} (u_\nu^+ -  u_\nu^-).
\end{align}
Note that $u \in L^2_{\loc} (\R; E)\cap H^1_{\loc}\big(\R; D(\Aell^{-1/2})\big)\cap H^2_{\loc}\big(\R; D(\Aell^{-1})\big)$ and the equation in \eqref{eq: abstract wave equation} is
fulfilled in $L_{\loc}^2\big(\R; D(\Aell^{-1})\big)$.
The energy of the solution is given by
\begin{align*}
    \E_0(u)(t) = \frac12 \big( \Norm{u(t)}{E}^2 +
\Norm{\d_tu(t)}{D(\Aell^{-1/2})}^2\big)
  = \frac12 \big( \Norm{u(t)}{E}^2 +
\Norm{\Aell^{-1/2} \d_tu(t)}{E}^2\big).
\end{align*}
 It is constant with respect to $t$,
that is,
\begin{align}
  \label{eq: semicl reduc energy 1}
  \E_0(u)(t)
  &= \E_0(u)(0)
    = \frac12 \big( \Norm{\udl{u}^0}{E}^2 
    + \Norm{\udl{u}^1}{D(\Aell^{-1/2})}^2\big)
    = \frac12\sum_{\nu \in \N}
    \big( |\udl{u}^0_\nu|^2 +\lambda_\nu^{-1}|\udl{u}^1_\nu|^2\big)\\
  &=\sum_{\nu \in \N}  \big(
    | u_\nu^+|^2 +  | u_\nu^-|^2 
    \big) .
    \notag
\end{align}
We thus simply write $\E_0(u)$ and one has
\begin{align}
  \label{eq: semicl reduc energy 2}
  \E_0(u) 
  =\frac12 (t_2-t_1)^{-1}\int_{t_1}^{t_2} \big( \Norm{u(t)}{E}^2 +
\Norm{\d_tu(t)}{D(\Aell^{-1/2})}^2\big) \, d t,
\end{align}
for any time inteval $[t_1,t_2]$, leading to a well defined energy if
only considering the solution $u$ in $L^2_{\loc} (\R;E)\cap H^1_{\loc}\big(\R; D(\Aell^{1/2})\big)\cap H^2_{\loc}\big(\R; D(\Aell^{-1})\big)$.

More generally, if $s \in \R$ and $\udl{u}^0 \in D(\Aell^{s/2})$ and $\udl{u}^1\in D(\Aell^{(s-1)/2})$, the unique solution to \eqref{eq: abstract wave equation}
in $\cap_k \Con^k(\R; D(\Aell^{(s-k)/2})$
is given by \eqref{eq: semicl reduc wave eq solution}, and one can
define the energy
\begin{align*}
    \E_{s}(u)(t) 
  = \frac12 \big( \Norm{u(t)}{D(\Aell^{s/2})}^2 +
  \Norm{\d_tu(t)}{D(\Aell^{(s-1)/2})}^2\big)
  = \frac12 \big( \Norm{\Aell^{s/2}  u(t)}{E}^2 +
  \Norm{\Aell^{(s-1)/2} \d_tu(t)}{E}^2\big),
\end{align*}
that is also constant with respect to $t$.
Note that if $u(t)$ is such a solution then $\Aell^{s/2} u(t)$ is a
solution  to \eqref{eq: abstract wave equation} in $\cap_k \Con^k\big(\R; D(\Aell^{-k/2})\big)$ as above,
with 
\begin{align}
  \label{eq: semicl reduc energy 3}
  \E_{s}(u) = \E_0 \big(\Aell^{s/2} u\big)
  =\sum_{\nu \in \N}  \lambda_\nu^{s}\big(
    | u_\nu^+|^2 +  | u_\nu^-|^2 \big) .
\end{align}
We shall say that such a solution to the wave equation lies in the $s$-energy
level. Similarly to \eqref{eq: semicl reduc energy 2} one has 
\begin{align}
  \label{eq: semicl reduc energy 4}
  \E_{s}(u) 
  =\frac12 (t_2-t_1)^{-1}\int_{t_1}^{t_2} \big( \Norm{u(t)}{D(\Aell^{s/2})}^2 +
\Norm{\d_tu(t)}{D(\Aell^{(s-1)/2})}^2\big) \, d t,
\end{align}
for any time inteval $[t_1,t_2]$.

\medskip
If $\udl{u}^0 , \udl{u}^1\in D(\Aell^\infty)$ the associated solution
$u(t)$ is such that $u \in \Con^k(\R; \Aell^s)$ for any $k \in \N$ and
$s \in \R$. One has $\E_s(u)< \infty$ and one says that $u(t)$ lies in
all energy levels.

\medskip
If $\udl{u}^0 \in D(\Aell^{\ell/2})$ and $\udl{u}^1\in
D(\Aell^{(\ell-1)/2})$ and if one denotes by $u(t)$ the unique
solution to the wave equation \eqref{eq: abstract wave
  equation} that lies in the $\ell$-energy level, 
there exists a sequence $u_n(t)$ of solutions that lie in all
energy levels and such that
\begin{align}
  \label{eq: density smooth solutions}
  \E_\ell (u - u_n) \to 0 \qquad \text{as} \ n \to +\infty,
\end{align}
from the density of $D(\Aell^\infty)$ in $D(\Aell^s)$ for any $s\in
\R$. It suffices to consider two sequences $(\udl{u}_n^0)_n$ and
$(\udl{u}_n^1)_n$ both in  $D(\Aell^{\infty})$ such that $\Norm{\udl{u}^0 - \udl{u}^0_n}{D(\Aell^{\ell/2})} \to 0$
and $\Norm{\udl{u}^1 - \udl{u}^1_n}{D(\Aell^{(\ell-1)/2})} \to 0$ and
let $u_n(t)$ be the associated solution to the wave equation.

\subsection{Dyadic decomposition for waves}

Let $ 0<\alpha<1$,  $ \varrho\in ]1, 1/\alpha[$ and set  
\begin{align*}
  J_k= \{\nu;\ \alpha \varrho^{|k|} 
  \leq {\sqrt {\lambda_\nu}}< \varrho^{|k|}/\alpha \}, \qquad k \in \ZZ^*.
\end{align*}
Note that $\# J_k < \infty$ from the assumed properties of the
eigenvalues. 
 Set also $\hk = \varrho^{-|k|}$. Introduce 
\begin{align*}
  E_k=\Span \{  e_\nu; \ \nu\in J_k \},
\end{align*}
equipped with the norm  $\Norm{u}{E} = \sum_{\nu\in J_k} | u_\nu
|^2$ for $u = \sum_{\nu \in J_k} u_\nu e_\nu \in E_k$.
Observe that if $u \in E_k$ then $\Aell^n u \in E_k$, using that $\# J_k < \infty$. Hence, $E_k$ is a
subspace of $D(\Aell^\infty)$.

At this stage it is important to  note that $J_{-k} = J_k$ implying $E_{-k} = E_k$. 
However, we shall identify $u \in E_k$ 
with the  following solution of the wave equation
\begin{align}
\label{eq: wave Ek}
 u=\sum _{\nu \in J_k} 
  e^{\sgn(k) i t{\sqrt {\lambda _\nu}}} u_\nu e_\nu.
\end{align}
The sign of $k$ here becomes important. 
Yet, note that $u \in E_k$ if
and only if $\bar{u} \in E_{-k}$ through this identification since
the eigenfunctions $e_\nu$ are assumed real. 

Following up, we identify $\d_t^\ell u$ with $u = \sum_{\nu \in J_k}
\big(i \sgn(k)\big)^\ell
\lambda_\nu^{\ell/2}  u_\nu e_\nu \in E_k$, that is, its value at $t=0$.
Similarly, one identifies $\Aell^s u$  with $ \sum_{\nu \in J_k} 
\lambda_\nu^{s}  u_\nu e_\nu \in E_k$.
\begin{lemma}
  \label{lemma: norm dt uk}
  For $u \in E_k$, the norms 
  \begin{align*} 
   \hk^{2s+r}   \Norm{\d_t^r  \Aell^s u}{E}, \quad r \in \N, \
    s\in \R, 
  \end{align*}
are equivalent to $\Norm{u}{E}$, uniformly with
  respect to $k\in \ZZ^*$.
\end{lemma}
\begin{proof}
  One writes
  \begin{align*} 
    \hk^{2(2s + r)} \Norm{\d_t^r  \Aell^s u}{E}^2 = 
    \sum_{\nu \in J_k} |\hk^2 \lambda_\nu|^{2s + r}\, 
    |u_\nu|^2 
    \eqsim 
    \sum_{\nu \in J_k}   |u_\nu|^2 
    =  \Norm{u}{E}^2, 
\end{align*}
as $\hk^2 \lambda_\nu\eqsim 1$ for $\nu \in J_{k}$.
\end{proof}
For $u \in E_k$, the identified solution to the wave equation given
in \eqref{eq: wave Ek} lies in all energy
level. One has 
\begin{align*}
  \E_{s}(u) = \frac12 \big( \Norm{\Aell^{s/2}  u(t)}{E}^2 +
  \Norm{\Aell^{(s-1)/2} \d_tu(t)}{E}^2\big)
  \eqsim \hk^{2s}   \Norm{u}{E}^2. 
\end{align*}
In particular, note that for $u \in E_k$ both terms in the
energy coincide; this is not the case in general for a
solution of the wave equation for fixed time (while it is true in time average). The reason is that $u \in E_k$ is in
fact solution to the following half-wave equation
\begin{align*}
  (\d_t - \sgn(k) i \Aell^{1/2}) u = 0.
\end{align*}

\bigskip
We introduce the
following sets of sequences of functions 
\begin{align} 
\label{eq: space B}
&B= \big\{ (u^k)_{k \in {\ZZ}^*};\ u^k \in
E_k  \ \text{and} \  \Norm{u^k}{L^2(\Omega)}  \leq 1 \big\}, \\
&B^\pm= \big\{ (u^k)_{k \in \pm{\N}^*};\  u^k \in
E_k \ \text{and} \  \Norm{u^k}{L^2(\Omega)}  \leq 1 \big\}.\notag
\end{align}

\subsection{Abstract Schr\"odinger equation and dyadic decomposition}
The Schr\"odinger equation associated with the operator $\Aell$ reads
\begin{align}
  \label{eq: abstract Schrodinger equation}
  \d_t u -i \Aell u = 0, \qquad u\bt=\udl{u}^0.
\end{align}
With the initial conditions $\udl{u}^0 \in D(\Aell^p)$, for some $p\in
\R$, the unique solution to \eqref{eq: abstract Schrodinger equation}
in $\Con^0\big(\R;  D(\Aell^p)\big)\cap \Con^1\big(\R; D(\Aell^{p-1})\big)$
is given by 
\begin{align}
  \label{eq: semicl reduc Schrodinger eq solution}
  u(t) = \sum_{\nu \in \N} e^{i t \lambda_\nu} \udl{u}^0_\nu\ e_\nu.
\end{align}
One has 
\begin{align*}
  \Norm{u(t)}{D(\Aell^p)}^2  
  =\sum_{\nu \in \N}   \lambda_\nu^{2 p} |\udl{u}^0_\nu|^2
  = \Norm{\udl{u}^0}{D(\Aell^p)}^2 .
\end{align*}
Moreover, one has $u \in \cap_k \Con^k\big(\R;  D(\Aell^{p-k})\big)$. 

\bigskip
As above let  $ 0<\alpha<1$,  $ \varrho\in ]1, 1/\alpha[$ and set
$\hk = \varrho^{-|k|}$ and 
\begin{align*}
  J^S_k= \{\nu;\ \alpha \varrho^{|k|} 
  \leq \lambda_\nu< \varrho^{|k|}/\alpha \}, \qquad k \in \N^*.
\end{align*}
Note that $\# J^S_k < \infty$ from the assumed properties of the
eigenvalues. 
Introduce 
\begin{align*}
  E^S_k=\Span \{  e_\nu; \ \nu\in J^S_k \},
\end{align*}
equipped with the norm  $\Norm{u}{E}^2 = \sum_{\nu\in J^S_k} | u_\nu
|^2$ for $u = \sum_{\nu \in J^S_k} u_\nu e_\nu \in E^S_k$.
Note that $E^S_k$ is a
subspace of $D(\Aell^\infty)$.

We shall identify $u \in E_k$ 
with the  following solution to the Schr\"odinger  equation
\begin{align}
\label{eq: Schrodinger Ek}
 u=\sum _{\nu \in J^S_k} 
  e^{i t \lambda _\nu} u_\nu e_\nu.
\end{align}

The counterpart to Lemma~\ref{lemma: norm dt uk} is the following lemma.
\begin{lemma}
  \label{lemma: norm dt uk Schrodinger}
  For $u \in E^S_k$, and $r \in \N$ and $s\in \R$  the norm
  \begin{align*} 
   \hk^{s+r}   \Norm{\d_t^r  \Aell^s u}{E}, \quad 
  \end{align*}
  is equivalent to $\Norm{u}{E}$, uniformly with respect to $k\in \N^*$.
\end{lemma}

\bigskip
We introduce the
following set of sequences of functions 
\begin{align} 
\label{eq: space B Schrodinger}
&B^S= \big\{ (u^k)_{k \in {\N}^*};\ u^k \in
E^S_k  \ \text{and} \  \Norm{u^k}{L^2(\Omega)}  \leq 1 \big\}.
\end{align}
If $\udl{u}^0\in D(\Aell^\infty)$ the associated solution
$u(t)$ is such that $u \in \Con^k(\R; \Aell^s)$ for any $k \in \N$ and
$s \in \R$. If $\udl{u}^0 \in D(\Aell^{p})$, denote by $u(t)$ the unique
solution to the Schr\"odinger equation \eqref{eq: abstract Schrodinger
  equation} that lies in $\Con^0\big(\R; D(\Aell^{p})\big)$. From the
density of $D(\Aell^\infty)$ in $D(\Aell^s)$ for any $s\in \R$ one can
consider
a sequence $(\udl{u}_n^0)_n \subset D(\Aell^{\infty})$ such that
$\Norm{\udl{u}^0 - \udl{u}^0_n}{D(\Aell^{p})} \to 0$. The associated
solutions $u_n(t)$ to the Schr\"odinger equation are such that 
\begin{align}
    \label{eq: density smooth solutions - Schrodinger}
  \sup_{t \in \R} \Norm{u(t) - u_n(t)}{D(\Aell^{p})} =\Norm{\udl{u}^0 - \udl{u}^0_n}{D(\Aell^{p})} \ \to 0 \qquad \text{as} \ n \to +\infty.
\end{align}

\section{Main results}

\subsection{Observation operator and unique continuation assumption}

For some Hilbert space $K$ consider an observation operator
$\Obs: E \to K$, possibly unbounded, with domain given by $D(\Obs) = D(\Aell^{m_0})$ for
some $m_0\in \R$, with 
\begin{align}
  \label{eq: bound obs op}
  \Norm{\Obs u}{K}\leq C_0 \Norm{u}{D(\Aell^{m_0})}.
\end{align}

 We introduce the following
assumption.
\begin{assumption}[unique continuation]
  \label{assumpt: unique continuation}
  If $u$ is an eigenvector of  $\Aell$
  such that $\Obs (u) =0$,  then $u =0$. 
\end{assumption}
Observe that an eigenvector of $\Aell$ lies in $D(\Aell^\infty)$ and thus lies in $D(\Obs)$.

\subsection{From semi-classical observation to observability for waves}

Our starting point will be  the following 
property.
\begin{scproperty}[wave equation]
For some  $\ell_1 \in \R$,  $C>0$, $k^0 \in \N$ and some $T>0$
one has 
\begin{equation}
  \label{eq: semicl obs}
  \E_{\ell_1}(u^k)
  \leq C \int _0^{T} \Norm{\Obs u^k(t)}{K}^2dt,
  \qquad (u^k)_{k\in \ZZ} \in B, \ |k| \geq k_0.
\end{equation}
\end{scproperty}

\medskip
Our main result in the case of the wave equation is the following theorem.
\begin{theorem}
  \label{theorem: main result - wave} 
  Let $\ell_1 \in \R$ with $\ell_1 \leq 2 m_0$.  
  Assume that there exists $C>0$, $ k_0>0$, and $T>0$ such that \eqref{eq: semicl obs}
  holds for any $U= (u^k)_{k\in \ZZ}\in B$ and  any $|k|\geq k_0$.
Under
the unique continuation 
Assumption~\ref{assumpt: unique continuation}, for any $T'>T$ there exists $ C'>0$ such
that for any $(\udl{u}^0, \udl{u}^1) \in D(\Aell^{m_0})
\times D(\Aell^{m_0-1/2})$ the solution to~\eqref{eq: abstract wave
  equation} given by \eqref{eq: semicl reduc wave eq solution} satisfies 
\begin{equation}
   \label{eq: observation wave}
\E_{\ell_1}(u)
  \leq C' \int_{0}^{T'} \Norm{\Obs u(t)}{K}^2\, d t. 
\end{equation}
\end{theorem}
Hence, semi-classical observation on a interval of length $T$ implies
classical observation on any interval of greater length.

Note that the \rhs in \eqref{eq: observation wave} makes sense because
of \eqref{eq: bound obs op} and $u(t) \in L^2_{\loc}\big(\R;
D(\Aell^{m_0})\big)$. Note that the requirement $\ell_1 \leq 2 m_0$ is natural since $u(t)$ lies in the $(2m_0)$-energy level.

\begin{remark}
\label{rem: comp conjugate}
  Let $\bar{u}$ denote the complex conjugate. 
In many cases one has
  \begin{align}
    \label{eq: complex conjugation}
    \Norm{\Obs(\bar{u})}{K} = \Norm{\Obs(u)}{K}.
  \end{align}
As $\overline{u^k} \in E_{-k}$ if $u^k \in
E_{k}$  note that having
\begin{equation}
  \label{eq: semicl obs B+}
  \E_{\ell_1}(u^k)
  \leq C \int _0^{T} \Norm{\Obs u^k(t)}{K}^2dt,
  \qquad (u^k)_{k\in \N} \in B^+, \ k\geq k_0,
\end{equation}
implies \eqref{eq: semicl obs}.
Consequently, if \eqref{eq: complex conjugation} does hold, then
assuming \eqref{eq: semicl obs B+} suffices to reach the conclusion of 
Theorem~\ref{theorem: main result - wave}.
\end{remark}

\bigskip
In the case $\ell_1 = 2m_0$, the argument we develop leading to a proof of
Theorem~\ref{theorem: main result - wave} is based
on \cite{Lebeau:92} (see also \cite{Burq:1997}), yet with more details
provided  here. The proof of Theorem~\ref{theorem: main result - wave} in this first case is given in
Section~\ref{sec: proof ell1 = ell0}. The argument is further refined to treat the case
$\ell_1 < 2 m_0$, that is, the case of an observability estimate
with some energy loss. The proof of Theorem~\ref{theorem: main result - wave} in this  second
case is carried out in Section~\ref{sec: proof ell1 < ell0}. Even
though the second case contains the first one, we chose to provide  a
simpler proof in the first case for the benefit of the reader.

\subsection{Admissibility condition for waves}

In the introduction we also considered admissibility conditions.
Such conditions are usefull in cases where $\Obs u$ makes sense in
energy levels lower than $2 m_0$. Note that the $(2m_0)$-level is
given by the boundedness of $\Obs$ on $D(\Aell^{m_0})$; see \eqref{eq: bound obs op}. Yet, since
$\Norm{\Obs u(t)}{K}$ appears in a time-integrated form in the
sought observability estimates, in some cases, one can expect
some improvement as 
formulated with the following additional assumption.
\begin{assumption}[admissibility condition for waves at the $\ell_0$-energy level]
  \label{assumpt: admissibility}
  For some $\ell_0 \leq 2 m_0$, the operator $\Obs$ extends as an unbounded
  operator from the subspace of $L^2_{\loc} (\R; D(\Aell^{\ell_0/2}))$ into
  $L^2_{\loc} (\R; K)$, also denoted by $\Obs$,  and for some $S>0$ and $C_S>0$
  one has 
  \begin{align}
    \label{eq: assumpt: admissibility}
  \int_0^S \Norm{\Obs u(t)}{K}^2 \,  d t \leq C_S \, \E_{\ell_0}(u),
\end{align}
for any 
$u \in L^2_{\loc} \big(\R; 
  D(\Aell^{\ell_0/2})\big)\cap H^1_{\loc}\big(\R; D(\Aell^{(\ell_0-1)/2})\big)\cap
  H^2_{\loc}\big(\R; D(\Aell^{\ell_0/2-1})\big)$ solution to \eqref{eq: abstract wave equation}.
\end{assumption}
In other words, Assumption~\ref{assumpt: admissibility} states that
$\Obs$ is bounded from the space of solutions that lie in the
$\ell_0$-energy level into $L^2(0, S; K)$. Considering only $\ell_0
\leq 2 m_0$ is natural since \eqref{eq: assumpt: admissibility} holds
for $\ell_0 = 2 m_0$ by  \eqref{eq: bound obs op}.

\begin{example}
A basic example where Assumption \eqref{assumpt: admissibility} is
usefull, meaning $\ell_0 < 2 m_0$, is the case of the  Dirichlet
Laplace operator $\Delta_D$ as in the introduction and the observation operator $\Obs$ given by the
Neumann trace operator localized in an open subset $\Gamma$ of $\d \Omega$, $\Obs u = \unitfunction{\Gamma} \d_\n u_{|\d\Omega}$,  as in Example~\ref{ex: observation
  operator}-\eqref{ex: observation operator: neumann trace}. With
the 
trace map $H^{1/2 +\eps}(\Omega) \to H^{\eps}(\d\Omega)$, one can use
$D(\Obs) = H^{3/2 + \eps} \cap H^1_0(\Omega)= D(A^{m_0})$ with $m_0 = 3/4 + \eps/2$,
for any $\eps>0$.  If $\udl{u}^0 \in H^1_0(\Omega)$ and $\udl{u}^1\in
L^2(\Omega)$ the associated weak solution to the wave equation lies in $\Con^0 \big(\R; H_0^1(\Omega)  \big)
    \cap \Con^1 \big(\R; L^2(\Omega)\big)$. One thus has $\nabla u \in
    \Con^0 \big(\R; L^2(\Omega)  \big)$, a regularity too low to allow
    one to apply
    the trace theorem to define $\d_\n u_{|\d\Omega} = (\n \cdot
    \nabla  u)_{|\d\Omega}$. However, because of the so-called hidden regularity for
such a solution to the wave equation, one finds that the trace $\d_\n
u_{|\d\Omega}$ makes sense and lies in $L^2_{\loc}(\R; L^2(\d\Omega))$; see for example \cite{LLT:86}. A
weak solution lies in the $1$-energy level we have defined and
moreover one has, for any $S>0$, 
\begin{align*}
  \int_0^S \Norm{\unitfunction{\Gamma} \d_\n u_{|\d\Omega}}{L^2(\d \Omega)}^2\,  d t
  \lesssim \E_1(u).
\end{align*}
In this case, one has $1 = \ell_0 < 2 m_0 = 3/2 + \eps$. 
\end{example}

\bigskip
From the time invariance of the
energy with \eqref{eq: assumpt: admissibility} one finds 
\begin{align*}
  \int_J \Norm{\Obs u(t)}{K}^2 \,  d t \leq C_S\,  \E_{\ell_0}(u),
\end{align*}
for any interval $J$ of length $|J|=S$. Moreover, for any
bounded interval $I$ one has 
\begin{align}
  \label{eq: admissibility condition}
  \int_I \Norm{\Obs u(t)}{K}^2 \,  d t \leq C_{|I|}\,  \E_{\ell_0}(u),
\end{align}
for some $C_{|I|} >0$ only function of $|I|$. 

With Assumption~\ref{assumpt: admissibility} one obtains the following
corollary to Theorem~\ref{theorem: main result - wave}.
\begin{corollary}
  \label{cor: main result - wave} 
    Let $\ell_1 \leq \ell_0 \leq 2 m_0$.  
  Assume that there exists $C>0$, $ k_0>0$, and $T>0$ such that \eqref{eq: semicl obs}
holds for any $U= (u^k)_{k\in \N}\in B^+$ and  any $k\geq k_0$.
Assume also  that \eqref{eq: complex conjugation} holds. Under
the unique continuation 
Assumption~\ref{assumpt: unique continuation} and the admissibility
Assumption~\ref{assumpt: admissibility} , for any $T'>T$ there exists $ C'>0$ such
that for any $(\udl{u}^0, \udl{u}^1) \in D(\Aell^{\ell_0/2})
\times D(\Aell^{(\ell_0-1)/2})$ the solution to~\eqref{eq: abstract wave
  equation} given by \eqref{eq: semicl reduc wave eq solution} satisfies 
\begin{equation}
   \label{eq: observation wave-cor}
\E_{\ell_1}(u)
  \leq C' \int_{0}^{T'} \Norm{\Obs u(t)}{K}^2\, d t. 
\end{equation}
\end{corollary}
The proof simply uses the density of solutions in the $(2m_0)$-energy
level in the space of solution in the $\ell_0$-energy
level and that both sides of the inequality \eqref{eq: observation
  wave-cor} are continuous with respect to the
$\ell_0$-energy; continuity of the \rhs is precisely \eqref{eq:
  admissibility condition} that follows from Assumption~\ref{assumpt: admissibility}.

A remark similar to Remark~\ref{rem: comp conjugate} can be made for
the result of Corollary~\ref{cor: main result - wave}.

\subsection{Main result for the Schr\"odinger equation}

\medskip
We first state what is meant by semi-classical observability in the case of
the Schr\"odinger equation.
\begin{scproperty}[Schr\"odinger equation]
For some  $p_1 \in \R$,  $C>0$, $k_{0} \in \N$ and some $T>0$
one has 
\begin{equation}
  \label{eq: semicl obs schrodinger}
  \Norm{u^k}{D(\Aell^{p_1})}
  \leq C \int _0^{T} \Norm{\Obs u^k(t)}{K}dt,
  \qquad (u^k)_{k\in \N} \in B^S, \ k \geq k_0.
\end{equation}
\end{scproperty}
\medskip
Our main result in the case of a the Schr\"odinger equation is the following theorem.
\begin{theorem}
  \label{theorem: main result Schrodinger}
  Let $p_1 \leq m_0$.  Assume that there
  exists $C>0$, $ k_0>0$, and $T>0$ such that \eqref{eq: semicl obs
    schrodinger} holds for any $U= (u^k)_{k\in \N}\in B^S$ and any
  $k\geq k_0$. Under the unique continuation Assumption~\ref{assumpt:
    unique continuation}, for any $T'>T$
  there exists $ C'>0$ such that for any
  $\udl{u}^0\in D(\Aell^{m_0})$ the solution to~\eqref{eq: abstract Schrodinger equation} given by \eqref{eq: semicl reduc Schrodinger eq solution}
  satisfies
\begin{equation}
   \label{eq: observation Schrodinger}
   \Norm{\udl{u}^0}{D(\Aell^{p_1})}
  \leq C' \int_{0}^{T'} \Norm{\Obs u(t)}{K}\, d t. 
\end{equation}
\end{theorem}
Recall that $m_0$ is as given by the continuity property \eqref{eq: bound obs op}
for $\Obs$.

Similarly to waves an admissibility assumption reads as follows.
\begin{assumption}[admissibility condition the Schr\"odinger equation
  in $D(\Aell^{p_0})$]
  \label{assumpt: admissibility Schrodinger}
  For some $p_0 \leq m_0 $, the operator $\Obs$ extends as an unbounded
  operator from the subspace of $L^2_{\loc} (\R; D(\Aell^{p_0}))$ into
  $L^2_{\loc} (\R; K)$, also denoted by $\Obs$,  and for some $S>0$ and $C_S>0$
  one has 
  \begin{align*}
  \int_0^S \Norm{\Obs u(t)}{K} \,  d t \leq C_S \, \Norm{\udl{u}^0}{D(\Aell^{p_0})},
\end{align*}
for any 
$u \in L^2_{\loc} \big(\R; 
  D(\Aell^{p_0})\big)\cap H^1_{\loc}\big(\R; D(\Aell^{p_0-1})\big)$ solution to \eqref{eq: abstract Schrodinger equation}.
\end{assumption}
Then, for any
bounded interval $I$ one has 
\begin{align}
  \label{eq: admissibility condition Schrodinger}
  \int_I \Norm{\Obs u(t)}{K} \,  d t \leq C_{|I|}\,  \Norm{\udl{u}^0}{D(\Aell^{p_0})},
\end{align}
for some $C_{|I|} >0$ only function of $|I|$. 

With Assumption~\ref{assumpt: admissibility Schrodinger} one obtains the following
corollary to Theorem~\ref{theorem: main result Schrodinger}.
\begin{corollary}
  \label{cor: main result Schrodinger}
  Let $p_1 \leq p_0 \leq m_0$.  Assume that there
  exists $C>0$, $ k_0>0$, and $T>0$ such that \eqref{eq: semicl obs
    schrodinger} holds for any $U= (u^k)_{k\in \N}\in B^S$ and any
  $k\geq k_0$. Under the unique continuation Assumption~\ref{assumpt:
    unique continuation} and the admissibility
  Assumption~\ref{assumpt: admissibility Schrodinger}, for any $T'>T$
  there exists $ C'>0$ such that for any
  $\udl{u}^0\in D(\Aell^{p_0})$ the solution to~\eqref{eq: abstract Schrodinger equation} given by \eqref{eq: semicl reduc Schrodinger eq solution}
  satisfies
\begin{equation}
   \label{eq: observation Schrodinger-cor}
   \Norm{\udl{u}^0}{D(\Aell^{p_1})}
  \leq C' \int_{0}^{T'} \Norm{\Obs u(t)}{K}\, d t. 
\end{equation}
\end{corollary}

\subsection{Existing and potential applications}

In the introduction, we considered the wave equation on an open set of
$\R^d$. This can be generalized to the manifold setting. Consider a
compact connected Riemannian manifold $\M$ of dimension $d$ with
boundary endowed with a metric $g = (g_{ij})$.
Introduce the elliptic operator $A = A_{\k,g}= \k^{-1} \divg (\k
\nablag)$, that is, in local coordinates
\begin{align}
  \label{eq: elliptic operator}
  A f = \k^{-1} (\det g)^{-1/2} 
  \sum_{1\leq i, j \leq d}  \d_{x_i}\big( 
  \k  (\det g)^{1/2} g^{i j}(x)\d_{x_j} f
  \big).
\end{align}
where $\k$ is a positive function on $\M$. 
The metric $g$ and the function $\k$ can be assumed $\Con^k$ with $k
\geq 1$ or
Lipschitz. 
The operator $A$  is unbounded on $E= L^2(\M)$. With the domain $D(A) =H^2(\M) \cap H_0^1(\M)$ one
finds that $A$ is selfadjoint, with respect to the $L^2$-inner product, and
$A$ is negative. With the elliptic operator $A$ one also defines
 the wave operator 
 \begin{equation}
   \label{wave}
   P = P_{\k,g}=\d_t^2 - A_{\k,g},
\end{equation}
and one can consider the associated homogeneous wave equation
\begin{align*}
 \begin{cases}
     P  \, u =0
     & \text{in}\ \R\times\M,\\
     u=0 &  \text{in}\ \R \times\d\M,\\
      u_{|t=0}  = \udl{u}^0, \ \d_t u_{|t=0}  = \udl{u}^1 
      & \text{in}\ \M.
  \end{cases}
\end{align*}

For an open set $\omega\subset \M$ one can consider the observation
operator with $K = L^2(\Omega)$ and the action on a solution to the
wave equation given by $\Obs_\omega  u = \unitfunction{\R\times \omega} \d_t u$. It maps a
weak-solution as above into $L_{\loc}^2(\R\times \omega)$. 
For an open set $\Gamma \subset \d\M$ one can consider the
observation operator  with $K = L^2(\Gamma)$  and the action on a solution to the
wave equation given by $\Obs_\Gamma u = \unitfunction{\R\times \Gamma} \d_\n u_{|\d\Omega}$,
where $\d_{\n}$ is normal derivative at the boundary.
In both cases the admissibility Assumption~\ref{assumpt:
  admissibility} holds as one has 
\begin{align*}
  \int_0^S \Norm{\Obs_\omega u(t)}{L^2(\omega)}^2 d t \lesssim
  \E_1(u),
\ \ \text{and} \ \  
\int_0^S \Norm{\Obs_\Gamma u(t)}{L^2(\Gamma)}^2 d t \lesssim
  \E_1(u),
\end{align*}
for a weak solution and  for some $S>0$.  The second property is in fact the so-called
hidden regularity property of waves; see \eg \cite{LLT:86}.  In both
cases Assumption~\ref{assumpt: unique
  continuation} holds with classical unique continuation results for
elliptic operators; see for instance \cite[Theorem 2.4]{Hoermander:83}
and \cite[Theorems 5.11 and
5.13]{LRLR:V1}.  The result of Corollary~\ref{cor: main result - wave}  thus
applies. It is used without loss of energy, that is, in the case
$\ell_1 = 1$, in \cite{Burq:1997} for a boundary observation in the
case of 
$\Con^2$-coefficients and in
\cite{BDLR1} for both types of observations in the case of
$\Con^1$-coefficients with also result for Lipschitz coefficients by a
perturbation argument.  In these references,
powerful tools of semi-classical analysis and semi-classical measures
are key to prove a semi-classical observability estimate as in
\eqref{eq: semicl obs}.

\medskip Here, we also treat the case of the loss of derivatives, that is,
if $\ell_1 < \ell_0$ in the assumed semi-classical observability
estimate \eqref{eq: semicl obs} and in the resulting observability
estimate in Theorem~\ref{theorem: main result - wave} and Corollary~\ref{cor: main result - wave}. Estimates with such
losses can be found in the literature. We refer for instance to the
work of F.~Fanelli and E.~Zuazua \cite{Fan-Zua}. Their result is in
the case of very rough coefficients (log-Lipschitz) and only concerns
the wave equation in 
one space dimension. Results in higher dimensions are open to our knowledge and the
study of such cases could benefit from the use of simpler
localized-in-frequency waves and their semi-classical setting. Though observation estimates with loss of derivatives are not so common for waves, they appear quite naturally for Schr\"odinger equations. See~\cite{Burq-93, Bu-Leb-92}  for such results in the presence of weak (hyperbolic) trapping,  or~\cite[Sections 6.4 \& 6.5]{BuZw} and~\cite[Section 4]{Tu-1} for the observability of Schr\"odinger on the square with an observation in (say) the vertical boundary.

\section{Time microlocalization}
\label{sec: Time microlocalization 1}
Let $\H$ be a Hilbert space, $\H = D(\Aell^s)$ for some $s\in \R$ or $\H= K$ in what
follows.  For a function $F \in \Cinfc(\R_+^*)$, with $F \geq 0$, set
$F_k(\tau) = F(\sgn(k) \tau)$ for $k\in \ZZ^*$, and consider the
operator $F_k(\hk D_t)$ that simply acts as a Fourier multipliers on
function of time $t$ with values in $\H$. Most often we shall write
$F_k^\H(\hk D_t)$ to keep explicit on which space the operator
acts. Since $F_k(\hk\tau)$ is bounded,  $F_k^\H(\hk D_t)$ maps $L^2(\R; \H)$ into
itself. It also maps $\S(\R; \H)$ (\resp $\S'(\R; \H)$) into
itself. We shall choose $F$ according to the following lemma.
\begin{lemma}
  \label{lemma: semicl reduc choice F}
  One can choose $F$ supported in $]\alpha, \alpha^{-1}[$ 
 with $\sum_{k\in \ZZ^*} F_k(\hk \tau)^2 \geq 1$ if $|\tau| \geq 1$. 
\end{lemma}
\begin{proof}
  Let $\alpha < a < \rho^{-1}<1$ and $F \in \Cinfc(]\alpha, \alpha^{-1}[)$ such
  that $F=1$ on $[a,a^{-1}]$.  Let $|\tau|\geq 1$.  With $\hk =
  \rho^{-|k|}$, one has
  $F_k(\hk \tau)=1$ if $a< \rho^{-|k|} |\tau| < a^{-1}$ and $\sgn(k) =
  \sgn(\tau)$, or equivalently
  \begin{align*}
  \frac{\ln(|\tau|) + \ln(a)}{\ln (\rho)}
  \leq  |k| \leq
   \frac{\ln(|\tau|) + \ln(a^{-1})}{\ln (\rho)}
    \ \ \text{and} \ \  
    \sgn(k) = \sgn(\tau).
  \end{align*}
  The difference between the two bounds is $2 \ln(a^{-1}) / \ln
  (\rho)>2$  and one has $\frac{\ln(|\tau|) + \ln(a^{-1})}{\ln (\rho)} >1$.
  Hence, there is at least one value of $k \in \ZZ^*$ such
that  $F_k(\hk \tau) = 1$. 
\end{proof}

\medskip
Let $\T>0$.
For $\indj \in \ZZ$ set 
\begin{align}
  \label{eq: def I ell}
  I_\indj = [\indj \T, (\indj+1)\T[.
\end{align}
Define $H_\H$ as the space of
functions $w \in L^2_{\loc}(\R; \H)$ such that
\begin{align}
  \label{eq: def H - abstract arg}
  \Norm{w}{H_\H} := \sup_{\indj \in Z} \Norm{1_{I_\indj} w}{L^2(\R; \H)} < \infty,
\end{align} 
that is, the space of uniformly locally $L^2$-bounded functions with
values in $\H$.

One has $H_\H \subset \S'(\R; \H)$ and thus $F_k^\H(\hk D_t)w$ is a well
defined tempered distribution in time $t$ with values in $\H$.  The following lemma
improves upon this result. 
\begin{lemma}
  \label{lemma: semicl reduc action on L2loc}
  The operator $F_k^\H(\hk D_t)$ fulfills the following properties.
  \begin{enumerate}
  \item One has 
    \begin{align}
      \label{eq: summation formula Fk}
      F_k^\H(\hk D_t) w (t) = \sum_{\indj \in \ZZ}   F_k^\H(\hk D_t)
      \big(1_{I_\indj} w\big),\qquad w \in H_\H,
    \end{align}
    and for any $\phi \in \Cinfc(\R)$, there exists $C>0$ such that 
    \begin{align}
      \label{eq: bound Linf Fk}
      \Norm{\phi F_k^\H(\hk D_t) w}{L^\infty(\R; \H)} \leq C  \Norm{w}{H_\H},
    \end{align}
    meaning that $F_k^\H(\hk D_t)$ maps $H_\H$ into $L_{\loc}^\infty(\R; \H)$ continuously.
    
  \item There exists $C>0$ such that 
    \begin{align*}
      \sum_{k \in \ZZ^*} \Norm{F_k^\H(\hk D_t)(\psi w)}{L^2(\R; \H)}^2 
      \leq C \Norm{\psi w}{L^2(\R; \H)}^2, 
    \end{align*}
    for $w \in H_\H$ and $\psi \in L^\infty(\R)$ with compact support. 
    \item If $\varphi \in \Cinfc(]0,\T[)$ and $\psi \in L^\infty(\R)$
      is such that $\psi =1$ in $I_0$, then for any $M\geq 1$,
      there exists $C_M>0$ such that 
      \begin{align}
        \label{eq: w -> psi w}
        \bigNorm{\varphi F_k^\H(\hk D_t) 
        \big((1- \psi) w\big) }{L^2(\R; \H)}
        \leq C_M \hk^M \Norm{w}{H_\H}.
      \end{align}
  \end{enumerate}
\end{lemma}

\medskip
\begin{proof}
  
Let $w \in H_\H$ and set $w_\indj = 1_{I_\indj} w$. One has 
\begin{align*}
  F_k^\H(\hk D_t) w_\indj (t)
  = \frac{1}{2 \pi} \iint e^{i t \tau}
    F_k(\hk \tau) \hat{w}_\indj (\tau) d  \tau
    = \frac{1}{2 \pi} \iint e^{i (t-s) \tau}
                               F_k(\hk \tau) w_\indj (s) d \tau d s.
\end{align*}
Note that $ F_k^\H(\hk D_t) w_\indj (t) \in \S(\R; \H)$ since its Fourier
transform in $t$, $F_k(\hk \tau) \hat{w}_\indj(\tau)$,  is  in
$\Cinfc(\R; \H)$.
One finds
\begin{align}
  \label{eq: basic estimate Fk wl}
  \Norm{F_k^\H(\hk D_t) w_\indj (t)}{\H}
  &\lesssim  \int_\R  F_k(\hk \tau) d \tau \, \int_\R \Norm{w_\indj (s)}{\H} d s
  \lesssim \T^{1/2} \hk^{-1} \Norm{F}{L^1} \Norm{w_\indj}{L^2(\R; \H)}\\
  &\lesssim \T^{1/2} \hk^{-1} \Norm{F}{L^1} \Norm{w}{H_\H},
    \qquad t \in \R. \notag
\end{align}
For the first point of the lemma we treat the case
$k>0$. The case $k<0$ can be treated similarly.
Consider $\phi \in
\Cinfc(\R)$ and $\indj\in\ZZ$ such that $\dist(\supp(\phi),I_\indj) \geq
\gamma >0$. Using that
\begin{align*}
  \frac{-i }{t-s}\d_\tau e^{i (t-s) \tau} = e^{i (t-s) \tau},
\end{align*}
for $t\neq s$, with $N$ integrations by parts one writes
\begin{align*}
  \phi(t) F_k^\H(\hk D_t) w_\indj (t)
  = \frac{i^N \hk^N}{2 \pi} \iint e^{i (t-s) \tau}
  F_k^{(N)}(\hk\tau) \frac{\phi(t) w_\indj (s)}{(t-s)^N} d \tau d s.
\end{align*}
One finds
\begin{align}
  \label{eq: decay estimate Fk}
  \Norm{\phi(t) F_k^\H(\hk D_t) w_\indj (t)}{\H}
  \lesssim \T^{1/2} \hk^{N-1}\gamma^{-N}\Norm{\phi}{L^\infty}
  \Norm{F^{(N)}}{L^1} \Norm{w}{H_\H}, \quad t\in \R.
\end{align}
If one chooses $N \geq 2$, 
with the $\gamma^{-N}$ factor the series $\sum_\indj F_k^\H(\hk D_t)
w_\indj (t)$ converges in $L_{\loc}^\infty(\R; \H)$.
Since $\sum_\indj F_k^\H(\hk D_t) w_\indj$ converges to $F_k^\H(\hk D_t) w$ in $\S'(\R; \H)$ one
concludes that $F_k^\H(\hk D_t) \in L_{\loc}^\infty(\R; \H)$ and that \eqref{eq: summation formula Fk} holds. One also concludes that the estimate
\eqref{eq: bound Linf Fk} holds using also
\eqref{eq: basic estimate Fk wl} for a finite numbers of terms.

\medskip 
Let now $\psi \in L^\infty$ have  compact support. Then $\psi
w \in L^2(\R; \H)$ and the Fourier transform of  $F_k^\H(\hk D_t) (\psi
w)$ is  $F_k(\hk \tau) \widehat{\psi w}(\tau)$ giving 
\begin{align*}
  \Norm{F_k^\H(\hk D_t) (\psi w)}{L^2(\R;\H)}^2
  = \int_\R F_k(\hk \tau)^2  \Norm{ \widehat{\psi w} (\tau)}{\H}^2\,  d \tau.
\end{align*}
Since $\supp(F) \subset [\alpha,\alpha^{-1}]$ and $\hk = \rho^{-|k|}$
one finds that $F_k(\hk \tau)\neq 0$ if $\tau \neq 0$ and 
\begin{align*}
  \frac{\ln(|\tau|) + \ln(\alpha)}{\ln (\rho)}
  \leq  |k| \leq
   \frac{\ln(|\tau|) + \ln(\alpha^{-1})}{\ln (\rho)}.
\end{align*}
The difference between the two bounds is $2 \ln(\alpha^{-1}) / \ln
(\rho)$. Most important, it is constant. Hence, the sum $\sum_k F_k(\hk \tau)^2$ only
involves a finite number $m$ of terms that is independent of
$\tau$. Consequently
\begin{align*}
  \sum_{k \in \ZZ^*} \Norm{F_k^\H(\hk D_t)(\psi w)}{L^2(\R; \H)}^2
  \leq
 m \Norm{F}{L^\infty}^2  \Norm{\widehat{\psi w}}{L^2(\R; \H)}^2
  \leq
  m \Norm{F}{L^\infty}^2
  \Norm{\psi w}{L^2(\R; \H)}^2.
\end{align*}

\medskip
Finally, consider $\varphi \in \Cinfc(]0,\T [)$ and $\psi \in L^\infty(\R)$
such that $\psi =1$ in $I_0$.
With \eqref{eq: summation formula Fk} one has
\begin{align*}
  F_k^\H(\hk D_t) \big( (1-\psi) w\big)
  =\sum_{|\indj| \geq 1}  F_k^\H(\hk D_t) \big( (1-\psi) w_\indj\big), 
  \qquad w_\indj = 1_{I_\indj} w.
\end{align*}
Let $N \geq 2$. 
For $|\indj|=1$ with 
\eqref{eq: decay estimate Fk} one obtains, 
\begin{align*}
  \Norm{\varphi(t) F_k^\H(\hk D_t) \big( (1-\psi) w_\indj\big) (t)}{L^\infty(\R;\H)}
  \leq C_N \T^{1/2} \hk^{N-1} \gamma_0^{-N} \Norm{\varphi}{L^\infty}
  \Norm{w}{H_\H},
\end{align*}
with $\gamma_0 = \dist\big( \supp(\varphi), I_0^c\big)$, using that  $\Norm{ (1-\psi) w}{H}\lesssim \Norm{w}{H_\H}$.
For $|\indj|\geq 2$ one finds in turn
\begin{align*}
  \Norm{\varphi(t) F_k^\H(\hk D_t) \big( (1-\psi) w_\indj\big) (t)}{L^\infty(\R;\H)}
  \leq C_N \T^{1/2} \hk^{N-1} \big( (|\indj|-1) \T\big)^{-N} \Norm{\varphi}{L^\infty}
  \Norm{w}{H_\H}.
\end{align*}
Since $\sum_{|\indj|\geq 2} \big( (|\indj|-1) \T\big)^{-N}$ converges one
concludes that \eqref{eq: w -> psi w} holds. 
\end{proof}

\subsection{Action on waves}
We now consider the action of  $F_k(\hk D_t)$ on a solution $u(t)$
of the abstract wave equation \eqref{eq: abstract wave equation} as
given by \eqref{eq: semicl reduc wave eq solution} that lies in the
$\ell$-energy level for some $\ell \in \R$.  In such case, if one uses $\H =
D(\Aell^{\ell/2})$, with \eqref{eq: semicl reduc energy 4} one sees that $u \in H_{D(\Aell^{\ell/2})}$ as defined in \eqref{eq: def H -
  abstract arg}. Thus, $F_k^{D(\Aell^{\ell/2})}(\hk D_t) u (t)$ makes
sense by Lemma~\ref{lemma: semicl reduc action on L2loc}. 
One has the following result. 
\begin{lemma}
  \label{lemma: semicl reduc action on y}
  Let $u(t)$ be a solution to \eqref{eq: abstract wave equation} that
  lies in the $\ell$-energy level for some $\ell \in \R$.  
  Let $k \in \ZZ^*$. 
  One has  $F_k^{D(\Aell^{\ell/2})}(\hk D_t) u (t) \in E_k$ and 
  \begin{align}
    \label{eq: semicl reduc action on y}
     F_k^{D(\Aell^{\ell/2})}(\hk D_t) u (t)
    &=\sum_{\nu \in J_k}
  \begin{cases}
    F_k \big( \hk \sqrt{\lambda_\nu}\big)
    e^{i t \sqrt{\lambda_\nu}} u_\nu^+ e_\nu
    & \text{if} \ k>0,\\
     F_k \big( -\hk \sqrt{\lambda_\nu}\big) e^{-i t \sqrt{\lambda_\nu}}
     u_\nu^- e_\nu
      & \text{if} \ k<0.\\
  \end{cases}
  \end{align}
\end{lemma}
\begin{proof}
Observe that the series in \eqref{eq: semicl reduc wave eq solution}
that defines $u(t)$ converges in the space $H_{D(\Aell^{\ell/2})}$. Hence, with the
first point in Lemma~\ref{lemma: semicl reduc action on L2loc} one finds
\begin{align*}
    F_k^{D(\Aell^{\ell/2})}(\hk D_t) u (t)
    &= \sum_{\nu \in \N} F_k^{D(\Aell^{\ell/2})}(\hk D_t) \big( e^{i t \sqrt{\lambda_\nu}} u_\nu^+
    + e^{-i t \sqrt{\lambda_\nu}} u_\nu^-\big) e_\nu.
\end{align*}
One has $F_k(\hk D_t) e^{i r t}  = F_k(\hk r)  e^{i r t}$, $r\in \R$; see for
instance  (18.1.27) in \cite{Hoermander:V3}. This gives
\begin{align*}
    F_k^{D(\Aell^{\ell/2})}(\hk D_t) u (t)
    &=\sum_{\nu \in J_k}
  \Big( F_k \big( \hk \sqrt{\lambda_\nu}\big) e^{i t \sqrt{\lambda_\nu}} u_\nu^+
    + F_k \big( - \hk \sqrt{\lambda_\nu}\big)  
      e^{-i t \sqrt{\lambda_\nu}} u_\nu^-\Big)e_\nu,
\end{align*}
  since $F_k(\hk \sqrt{\lambda_\nu}\big) =0$ unless $\nu \in
  J_k$. This gives the result using the dependency of the support of $F_k$ upon
  the sign of $k$. 
\end{proof}

\bigskip
Because of the form of $u^k(t)=  F_k^{D(\Aell^{\ell/2})}(\hk D_t) u (t)$ one sees
that $u^k(t)$ is also solution to the wave equation.
Yet, as the sum is
finite  in \eqref{eq: semicl reduc action on y} one has
\begin{align}
  \label{eq: regularity yk}
   u^k(t)  \in \Con^m\big(\R; D(\Aell^r)\big), \quad \forall m\in
  \N, \ r \in \R,
\end{align}
that is, the wave $u^k(t)$ lies in all energy levels.

\medskip
We now consider the particular case of a solution $u(t)$ that lies in
the $(2m_0)$-energy level with $m_0$ as appearing in the continuity
property~\eqref{eq: bound obs op} of $\Obs$. 
\begin{lemma}
  \label{lemma: semicl filter}
  Let $u(t)$ be a solution to \eqref{eq: abstract wave equation} that
  lies in the $m_0$-energy level. 
  One has   $\Obs u^k  =  F_k ^{K} (\hk D_t) \Obs u$ in $L^\infty_{\loc}
  (\R; K)$ for $k \in \ZZ^*$.
\end{lemma}
\begin{proof}
  We treat the case $k >0$; the proof for the case $k<0$ is
  similar. With Lemma~\ref{lemma: semicl reduc action on y} one has
  $\Obs u^k
    = \sum_{\nu \in J_k}  F_k \big( \hk \sqrt{\lambda_\nu}\big)
    e^{i t \sqrt{\lambda_\nu}} u_\nu^+ \Obs e_\nu$,
  using that the sum is finite.
  For $n\in \N^*$ set
  \begin{align*}
    U_n(t) = \sum_{\lambda_\nu \leq n} \big( e^{i t \sqrt{\lambda_\nu}} u_\nu^+
    + e^{-i t \sqrt{\lambda_\nu}} u_\nu^-\big) e_\nu.
  \end{align*}
  One has 
  \begin{align*}
   \Norm{ U_n(t) -u(t)}{D(\Aell^{m_0})}^2  
    \lesssim \sum_{\lambda_\nu > n} 
    \lambda_\nu^{2m_0}\big( |u_\nu^+|^2
    + | u_\nu^-|^2 \big)^2  = \E_{2 m_0} ( u - U_n), \qquad t \in \R.
\end{align*}
  One has $\E_{2 m_0} ( u - U_n) \to 0$ as $n \to +\infty$ since
  $\big(\lambda_\nu^{m_0} u_\nu^\pm\big)_\nu \in \ell^2(\C)$. Hence, 
  \begin{align*}
   \Norm{ U_n -u }{H_{D(\Aell^{m_0})}} \lesssim 
    \Norm{ U_n -u }{L^\infty(\R; D(\Aell^{m_0})}  \to 0, 
    \ \ \text{as} \ \  n \to +\infty,
\end{align*}
yielding in turn
  \begin{align}
    \label{eq: lemma: semicl filter1}
    F_k^K (\hk D_t) \Obs  U_n
    \ \mathop{\longrightarrow}_{\n \to \infty} \
    F_k^K (\hk D_t) \Obs u \quad
    \text{in}\   L_{\loc}^\infty(\R; K), 
  \end{align}
  by Lemma~\ref{lemma: semicl reduc action on L2loc}.
  As the sum defining $U_n$ is finite one has
   \begin{align*}
     F_k ^K (\hk D_t) \Obs  U_n
     = \sum_{\lambda_\nu \leq n}  F_k(\hk  \sqrt{\lambda_\nu})
     e^{i t \sqrt{\lambda_\nu}} u_\nu^+ \Obs e_\nu,
   \end{align*}
   using the support property of $F_k$ for $k >0$ and that $F_k(\hk D_t) e^{i r t}  = F_k(\hk r)  e^{i r t}$, $r\in \R$; see for
   instance  (18.1.27) in \cite{Hoermander:V3}. One observes that
   $F_k^K (\hk D_t) \Obs  U_n = \Obs u^k$ for $n$ chosen \suff large. The
   limit in \eqref{eq: lemma: semicl filter1} hence gives the result.
\end{proof}

\subsection{Action on solutions to the Schr\"odinger equation}

The counterpart results of Lemmata~\ref{lemma: semicl reduc action on
  y} and \ref{lemma: semicl filter}
for the  Schr\"odinger equation are the following ones.  
\begin{lemma}
  \label{lemma: semicl reduc action on y Schrodinger}
  Let $u(t)$ be a solution to the Schr\"odinger equation~\eqref{eq:
    abstract Schrodinger equation} with $\udl{u}^0 \in D(\Aell^p)$ for
  $p \in \R$.  
  Let $k \in \N^*$. 
  One has  $F_k^{D(\Aell^{p})}(\hk D_t) u (t) \in E^S_k$ and 
  \begin{align}
    \label{eq: semicl reduc action on y Schrodinger}
     F_k^{D(\Aell^{p})}(\hk D_t) u (t)
    &=\sum_{\nu \in J^S_k}
  F_k \big( \hk \lambda_\nu\big)
    e^{i t \lambda_\nu} \udl{u}^0_\nu e_\nu.
  \end{align}
\end{lemma}
\begin{lemma}
  \label{lemma: semicl filter Schrodinger}
  Let $u(t)$ be a solution to \eqref{eq: abstract Schrodinger equation}  with $\udl{u}^0 \in D(\Aell^{m_0})$.
  One has   $\Obs u^k  =  F_k ^{K} (\hk D_t) \Obs u$ in $L^\infty_{\loc}
  (\R; K)$ for $k \in \N^*$ and $u^k(t)=  F_k^{D(\Aell^{m_0})}(\hk D_t) u (t)$.
\end{lemma}
The proof of Lemmata~\ref{lemma: semicl reduc action on
  y} and \ref{lemma: semicl filter} can be adapted {\em mutatis
  mutandis}. 

\bigskip
One sees
that $u^k(t)$ is also solution to the Schr\"odinger  equation and 
\begin{align}
  \label{eq: regularity yk Schrodinger}
   u^k(t)  \in \Con^m\big(\R; D(\Aell^r)\big), \quad \forall m\in
  \N, \ r \in \R,
\end{align}
as the sum is
finite  in \eqref{eq: semicl reduc action on y Schrodinger}.

\section{Proof of the main result for waves}

As explained below Theorem~\ref{theorem: main result - wave}, for the
benefit of the reader, we have
chosen to provide a proof for the case $\ell_1 =2 m_0$ and a proof for
the case $\ell_1 \leq 2 m_0$. Even though the second case contains the
first one, the proof is the first case is less technical.

\subsection{Case $\ell_1 =2 m_0$.}
\label{sec: proof ell1 = ell0}
Let $\T = (T+T')/2$ and $\delta^0 = (\T-T)/2 = (T'-\T)/2$. 
Because of the time invariance of the energy, the assumed
semi-classical observation inequality \eqref{eq: semicl obs} reads 
\begin{equation}
  \label{eq: semicl obs B-reformulated}
  \E_{\ell_1}(u^k)
  \leq C \int _{\delta^0}^{\T-\delta^0} \Norm{\Obs u^k(t)}{K}^2dt,
  \qquad (u^k)_{k\in \ZZ} \in B, \ |k| \geq k_0,
\end{equation}
and we aim to prove that, for any $\delta \in ]0, \delta^0]$,
\begin{equation*}
   \E_{\ell_1}(u)
  \leq C' \int_{-\delta}^{\T+\delta} \Norm{\Obs u(t)}{K}^2\, d t 
\end{equation*}
holds for any solution $u$ to the wave equation \eqref{eq: abstract wave
  equation} writen in \eqref{eq: semicl reduc wave eq solution} that
lies in the $\ell_1$-energy level, that is, $\udl{u}^0\in
D(\Aell^{m_0})$ and $\udl{u}^1\in D(\Aell^{m_0-1/2})$ here. 
The simultaneous treatment of $0< \delta \leq \delta^0$ is used for a technical
argument in the proof of Lemma~\ref{lemma: invisible solution} below.

For such a solution $u$, 
one notes that $\Obs u \in H_K$ by \eqref{eq: admissibility condition} and
one has 
\begin{align}
  \label{eq: observation in H}
  \Norm{\Obs u}{H_K}^2
  \lesssim \Norm{\Obs u}{L^\infty(\R; K)}^2
  \lesssim \Norm{u}{L^\infty(\R; D(\Aell^{m_0})}^2  
  \lesssim \E_{\ell_1}( u),
\end{align}
since $\ell_1 = 2 m_0$. 
Let $F$ be chosen as in Lemma~\ref{lemma: semicl reduc choice
  F}. 
With Lemma~\ref{lemma: semicl reduc
  action on y}, for $k \in \ZZ^*$, set
\begin{align}
  \label{eq: chi k y}
  u^k (t)= F_k^{D(\Aell^{\ell_1/2})}(\hk D_t) u(t)
  = \sum_{\nu \in J_k}
  \begin{cases}
    F_k \big( \hk \sqrt{\lambda_\nu}\big)
    e^{i t \sqrt{\lambda_\nu}} u_\nu^+ e_\nu
    & \text{if} \ k>0,\\
     F_k \big( -\hk \sqrt{\lambda_\nu}\big) e^{-i t \sqrt{\lambda_\nu}}
     u_\nu^- e_\nu
      & \text{if} \ k<0.
  \end{cases}
\end{align}
One has $u^k \in E_k$.
With 
the semi-classical observation property~\eqref{eq: semicl obs B-reformulated} one has
\begin{equation}
  \label{eq: semicl obs-bis}
  \E_{\ell_1}(u^k) 
  \lesssim \Norm{\varphi \Obs u^k(t)}{L^2(\R,K)}^2, 
  \qquad \text{for} \ |k| \geq k_0,
\end{equation}
where   $\varphi \in \Cinfc(]0,\T[)$ with $\varphi =1$ on a \nhd of 
$[\delta^0,\T-\delta^0]$.
One has 
\begin{align*}
  \E_{\ell_1}(u^k)
  = \sum_{\nu \in J_k}\lambda_\nu^{\ell_1}
  \begin{cases}
    F_k\big(\hk\sqrt{\lambda_\nu}\big)^2 |u^+_\nu|^2
    & \text{if} \ k>0,\\
  F_k\big(-\hk\sqrt{\lambda_\nu}\big)^2 |u^-_\nu|^2
  & \text{if} \ k<0.
  \end{cases}
\end{align*}
Set $u^0 = \sum_{\lambda_\nu \leq 1}
  \Big( e^{i t \sqrt{\lambda_\nu}} u_\nu^+
    + 
  e^{-i t \sqrt{\lambda_\nu}} u_\nu^-\Big)e_\nu$.
With Lemma~\ref{lemma: semicl reduc choice F} one finds
\begin{align*}
  \E_{\ell_1}(u- u^0)
  &= \E_{\ell_1}(u) -\E_{\ell_1}(u^0)
    =\sum_{\lambda_\nu>1} \lambda_\nu^{\ell_1}
    \big( |u^-_{\nu}|^2 + |u^+_{\nu}|^2\big)\\
  & \lesssim
    \sum_{k \in \ZZ^*} \sum_{\lambda_\nu>1}
    \lambda_\nu^{\ell_1} F_k\big(- \hk \sqrt{\lambda_\nu}\big)^2 |u^-_{\nu}|^2
    +  \sum_{k \in \ZZ^*} \sum_{\lambda_\nu>1}
    \lambda_\nu^{\ell_1} F_k\big( \hk \sqrt{\lambda_\nu}\big)^2 |u^+_{\nu}|^2\\
  & \lesssim
    \sum_{k \in -\N^*} \sum_{\nu \in \N}
    \lambda_\nu^{\ell_1} F_k\big(- \hk \sqrt{\lambda_\nu}\big)^2 |u^-_{\nu}|^2
    +  \sum_{k \in \N^*} \sum_{\nu \in \N}
    \lambda_\nu^{\ell_1} F_k\big( \hk \sqrt{\lambda_\nu}\big)^2 |u^+_{\nu}|^2\\
    & \lesssim \sum_{k \in \ZZ^*}  \E_{\ell_1}(u^k).
\end{align*}
One thus obtains with \eqref{eq: semicl obs-bis} 
\begin{align}
  \label{eq: separation energy}
  \E_{\ell_1}(u) &\lesssim  \E_{\ell_1}(u^0) + \sum_{k \in \ZZ^*}  \E_{\ell_1}(u^k)\\
  &\lesssim  \E_{\ell_1}(u^0) + \sum_{1 \leq |k| < k_1}  \E_{\ell_1}(u^k)
    + \sum_{|k|\geq k_1}\Norm{\varphi \Obs u^k(t)}{L^2(\R;K)}^2,
    \notag
\end{align}
for $k_1 \geq k_0$ to be chosen below.

With  Lemma~\ref{lemma: semicl filter}, one can write
\begin{align}
  \label{eq: semicl reduc  estimation energy}
  \E_{\ell_1}(u)
  &\lesssim  \sum_{0 \leq  |k| < k_1}  \E_{\ell_1} (u^k)
  +  \sum_{|k|\geq k_1}\Norm{\varphi F_k^{K}(\hk D_t) \Obs u}{L^2(\R;K)}^2.
\end{align}
Set $\psi_\delta = 1_{[-\delta,\T+\delta]}$. With the third point of Lemma~\ref{lemma: semicl reduc action on
  L2loc} and \eqref{eq: observation in H} one has
\begin{align*}
  \Norm{\varphi(t) F_k^{K} (\hk D_t) \Obs u}{L^2(\R; K)}^2
  &\lesssim
    \Norm{\varphi(t) F_k^{K} (\hk D_t) (\psi_\delta  \Obs u)}{L^2(\R; K)}^2
    +\hk^{2M} \Norm{\Obs u}{H_K}^2\\
  &\lesssim
    \Norm{\varphi(t) F_k^{K} (\hk D_t) (\psi_\delta  \Obs u)}{L^2(\R; K)}^2
    + \hk^{2M} \E_{\ell_1}(u).
\end{align*}
With the second point of lemma~\ref{lemma: semicl reduc action on
  L2loc} one finds
\begin{align*}
  \sum_{|k| \geq k_1} \Norm{\varphi(t) F_k^{K}(\hk D_t) \Obs u}{L^2(\R; K)}^2
  &\lesssim  \Norm{\psi_\delta  \Obs u}{L^2(\R; K)}^2
    + \sum_{|k| \geq k_1}h_{k}^{2M} \E_{\ell_1}(u)\\
  &\lesssim   \Norm{\psi_\delta  \Obs u}{L^2(\R; K)}^2
    +h_{k_1}^{2M} \E_{\ell_1}(u).
\end{align*}
using that $\hk = \rho^{-|k|}$ with $\rho>1$.
With \eqref{eq: semicl reduc  estimation energy} one finds
\begin{align*}
 \E_{\ell_1}(u)
  &\lesssim  \sum_{0 \leq  |k| < k_1}  \E_{\ell_1}(u^k)
    +  \Norm{\psi_\delta  \Obs u}{L^2(\R; K)}^2
    + h_{k_1}^{2M} \E_{\ell_1}(u).
\end{align*}
For $k_1\geq k_0$ chosen \suff large one obtains
\begin{align}
  \label{eq: semicl reduc  estimation energy2}
  \E_{\ell_1}(u)
  &\lesssim  \sum_{0 \leq  |k| < k_1}  \E_{\ell_1}(u^k)
  +  \Norm{\psi_\delta  \Obs u}{L^2(\R; K)}^2.
\end{align}
To remove the first term  on the \rhs of \eqref{eq: semicl reduc
  estimation energy2} we shall use the following lemma that states
that only the trivial solution is invisible for the observation operator
$\Obs$. 
\begin{lemma}[absence of invisible waves]
  \label{lemma: invisible solution}
  Let $u \in \cap_k \Con^k\big(\R; D(\Aell^{m_0-k/2})\big)$
be solution to
  \eqref{eq: abstract wave equation} and such that $\psi_\delta  \Obs u =0$.
  Then $u=0$.
\end{lemma}
Recall that writing $\psi_\delta  \Obs u =0$ makes sense by \eqref{eq: observation in H}.
\begin{proof}
  For $0< \delta\leq \delta^0$ as above, set 
  \begin{align*}
    \mathscr N_\delta = \big\{ u \in \cap_k \Con^k\big(\R;
    D(\Aell^{m_0-k/2})\big); \ u \ \text{solution to} \ \eqref{eq: abstract
    wave equation} 
    \ \text{and} \ \Obs u(t) =0 \ 
    \text{if} \ t\in ]-\delta,\T+\delta[
    \big\},
  \end{align*}
  that is, 
  the space of invisible solutions in the sense of the
  observation operator $\psi_\delta \Obs$.
  We equip $\mathscr N_\delta$ with the norm associated with the energy
  $\E_{\ell_1}$. With \eqref{eq: semicl reduc  estimation energy2} one has
  \begin{align}
    \label{eq: energy estimate N}
    \E_{\ell_1}(u)
    \lesssim  \sum_{0 \leq  |k| < k_1} \E_{\ell_1}(u^k)
    \ \ \text{on} \ \mathscr N_\delta, \quad 0< \delta\leq \delta^0.
  \end{align}
  As the maps $u \mapsto
  u^k$ have a finite rank, they are compact. With \eqref{eq: energy
    estimate N} it follows that $\mathscr
  N_\delta$ has a compact unit ball and is thus finite dimensional by
  the Riesz theorem.

  We claim that
  \begin{align}
    \label{eq: claim N}
    u \in \mathscr N_\delta
    \ \ \imp \ \
    u \in \bigcap_{m, r \in \N}
    \Con^r\big(\R; D(\Aell^{m})\big)
    \ \ \text{and} \ \ \d_t u \in \mathscr N_\delta.
  \end{align}
  The finite
  dimensional space $\mathscr N_\delta$ is thus stable under the
  action of the operator $\d_t$.  Consequently this operator has an
  eigenvector $\mathsf v \in \mathscr N_\delta$ with associated eigenvalue $\mu$. One finds
  $\Aell \mathsf v = \d_t^2 \mathsf v = \mu^2 \mathsf v$ meaning that
  $\mathsf v(t)$ is an eigenfunction for $\Aell$ for all $t \in  \R$. As $\Obs \mathsf v(t) =0$ if $t \in  ]-\delta,
  \T+\delta[$, with the
  unique continuation Assumption~\ref{assumpt: unique continuation}
  one obtains $\mathsf v(t)=0$ for all $t \in ]-\delta,
  \T+\delta[$. Hence, $\mathsf v=0$ since the energy of this solution is zero and one
  concludes that  $\mathscr N_\delta =\{0\}$.

  \medskip We now prove our claim~\eqref{eq: claim N}. Let $0 < \delta'
  < \delta$ and note that $\mathscr N_\delta \subset \mathscr
  N_{\delta'}$. Let $u \in \mathscr N_\delta$. For $0< \eps < \delta-
  \delta'$, observe that 
  \begin{align*}
    w_\eps(t) = \frac{u(t+\eps) - u(t)}{\eps} \in  \mathscr N_{\delta'},
  \end{align*}
 On the one hand, as $u \in \Con^1\big(\R; D(\Aell^{m_0-1/2})\big)$, one has
  \begin{align}
    \label{eq: time derivative N}
    w_\eps(t) \ \to\ \d_t u(t) \quad \text{in} \  D(\Aell^{m_0 -1/2}),
    \quad \forall t\in \R.
  \end{align}
  On the other hand, if one applies the operator $F_k^{D(\Aell^{m_0})}(\hk D_t)$,
  one has  
  \begin{align*}
    w^k_\eps(t) = \frac{u^k(t+\eps) - u^k(t)}{\eps}.
  \end{align*}
  Note indeed that  $F_k^{D(\Aell^{m_0})}(\hk D_t) (u(.+\eps)) (t) =
  F_k^{D(\Aell^{m_0})}(\hk D_t) u (t+\eps) = u^k(t+\eps)$ since $
  F_k^{D(\Aell^{m_0})}(\hk D_t)$ is a simple Fourier multiplier.
  Since $u^k
  \in \cap_{m,r \in \N}\Con^r \big(\R, D(\Aell^m)\big)$,  one finds
  that, for any $k$, $w^k_\eps(t)$ converges to $\d_t u
  ^k(t)$ in $\Con^{r} \big(J, D(\Aell^{m})\big)$,  for any $r, m \in
  \N$ and any bounded interval $J$. 
  Hence, recalling that $w^k_\eps$ and $\d_t u^k$ are
  solutions to the wave equation,  one finds $w^k_\eps \to \d_t u ^k$ in the norm associated with the
  $\ell_1$-energy. With \eqref{eq: energy estimate N} one finds that
  $(w_\eps)_\eps$ is of Cauchy type in $\mathscr N_{\delta'}$ for this latter norm, as $\eps \to 0$. It thus converges to some $w
  \in \mathscr N_{\delta'}$, as $\mathscr N_{\delta'}$ is complete
  since finite dimensional. Then, one has
  \begin{align*}
    w_\eps(t) \ \to w(t) \quad \text{in} \  D(\Aell^{m_0}), \quad \
    \text{uniformly for}\ t\in \text{any bounded interval of}\ \R,
  \end{align*}
  yielding $w = \d_t u$ by \eqref{eq: time derivative N}. Consequently $\d_t u\in \mathscr
  N_{\delta'}$ meaning $\d_t u \in \cap_k \Con^k\big(\R;
    D(\Aell^{m_0-k/2})\big)$ and $\Obs \d_t u (t) = 0$ for $t
    \in ]-\delta', \T+\delta'[$. Our choice of $\delta' \in
    ]0,\delta[$ is however arbitrary. Hence, one obtains $\Obs \d_t u (t) = 0$ for $t
    \in ]-\delta, \T+\delta[$, meaning that $\d_t u \in  \mathscr
  N_{\delta}$.

  Iterating the argument, one obtains
  $\Aell u = - \d_t^2 u \in  \mathscr
  N_{\delta}$, thus $\Aell u \in \cap_k \Con^k\big(\R;
    D(\Aell^{m_0-k/2})\big)$ implying 
    $u \in \cap_k \Con^k\big(\R;
    D(\Aell^{m_0+1-k/2})\big)$. Iterations give 
    $u \in \cap_{r, m \in \N} \Con^r\big(\R;
    D(\Aell^{m})\big)$.
\end{proof}

\medskip
We now conclude the proof of Theorem~\ref{theorem: main result - wave}
by a classical  argument by contradiction,
assuming that the observation inequality 
\begin{align}
  \label{eq: semicl reduc final observation}
  \E_{\ell_1}(u)
  &\lesssim  \Norm{\psi_\delta  \Obs u}{L^2(\R; K)}^2
\end{align}
does not hold.  Then, there exists a sequence of initial conditions
$(\udl{u}^{0,n}, \udl{u}^{1,n}) \in D(\Aell^{m_0}) \times
D(\Aell^{m_0-1/2})$ with associated solutions
$(u_n)_{n\in \N}$ to the wave equation such that $\E_{\ell_1}(u_n)=1$
and $\Norm{\psi_\delta  \Obs u_n}{L^2(\R; K)} \to 0$.  Some subsequence,
that we also write $(\udl{u}^{0,n}, \udl{u}^{1,n})$ for simplicity,
weakly converges to some
$(\udl{u}^0, \udl{u}^1) \in D(\Aell^{m_0}) \times
D(\Aell^{m_0-1/2})$. Associated with $(\udl{u}^0,\udl{u}^1) $ is
a solution $u$, also in the $\ell_1$-energy level, 
and $u_{n}$ converges weakly to $u$ in
$L^2\big(-\delta,\T+\delta; D(\Aell^{m_0})\big) \cap H^1\big(-\delta,\T+\delta;
D(\Aell^{m_0-1/2})\big)$. Moreover one has $\psi_\delta  \Obs u=0$. In fact,
one considers
$\tilde{\Obs}: D(\Aell^{m_0}) \times D(\Aell^{m_0-1/2}) \to
H_K$ with $\tilde{\Obs}(\udl{v}^0,\udl{v}^1) = \psi_\delta  \Obs v$
where $v$ is the linear wave with initial conditions $\udl{v}^0$ and
$\udl{v}^1$ as given by \eqref{eq: semicl reduc wave eq
  solution}. With~\eqref{eq: observation in H} the map
$\tilde{\Obs}$ is continuous. It is thus also continuous for the weak
topologies; see for instance \cite[Proposition 35.8]{Treves:67}. Since
$\tilde{\Obs}(\udl{u}^{0,n}, \udl{u}^{1,n})$ converges strongly to
$0$, and thus also weakly, this gives
$\tilde{\Obs}(\udl{u}^0, \udl{u}^1) =0$, that is, $\psi_\delta  \Obs u=0$.
With lemma~\ref{lemma: invisible solution} one concludes that
$u=0$, and thus $\udl{u}^0=\udl{u}^1=0$.

As above, for a linear wave  $v$ with initial conditions $(\udl{v}^0, \udl{v}^1)\in D(\Aell^{m_0}) \times D(\Aell^{m_0-1/2})$,
one observes that $(\udl{v}^0, \udl{v}^1) \mapsto v^k (t) =
F_k^{D(\Aell^{m_0})} (\hk D_t) v(t)$ is compact since with a finite
dimensional range; see the expression in Lemma~\ref{lemma: semicl reduc action on y}. As one has $(\udl{u}^0_n, \udl{u}^1_n)
\rightharpoonup  (0, 0)$, one obtains that $u_n^k$ converges strongly to $0$ in the
norm given by the $\E_{\ell_1}$-energy, for $0 \leq  |k| < k_1$.
Here, one thus obtains
\begin{align*}
  \lim_{n \to \infty}  \sum_{0 \leq  |k| < k_1}  \E_{\ell_1}(u_{n}^k) =0.
\end{align*}
Estimate~\eqref{eq: semicl
  reduc  estimation energy2} applied to $u_n$ thus leads to a
contradiction since both terms on the \rhs converge to zero and the
\lhs is equal to $1$. This concludes the contradiction argument and 
the proof of Theorem~\ref{theorem: main result - wave} in the case $\ell_1 =
2 m_0$.
\hfill \qedsymbol \endproof

\subsection{Refined time-microlocalization estimates}
\label{sec: Refined time-microlocalization estimates}
Here, we consider a solution $u(t)$ to the abstract wave equation \eqref{eq: abstract
  wave equation} with 
\begin{align*}
  \udl{u}^0 \in  D(\Aell^{m})
  \ \ \text{and} \ \  
  \udl{u}^1  \in D(\Aell^{m-1/2}) ,
\end{align*}
for some $m \in \R$. 
Then, $u(t)$ lies in the $(2m)$-energy level and $u \in \cap_k \Con^k(\R; D(\Aell^{m-k/2})$. 

Let $F \in \Cinfc(]\alpha, \alpha^{-1}[)$ be as given by Lemma~\ref{lemma: semicl
  reduc choice F}. Consider $\tF \in \Cinfc(]\alpha, \alpha^{-1}[)$ such that $\tF=1$
in a \nhd of $\supp(F)$. A first result
we shall use is the following one.
\begin{lemma}
  \label{lemma: refined estimate 1}
   Let $k \in \ZZ^*$.  
   Let $\ell\in \R$. There exists $C= C_{m, \ell}>0$ such that 
      \begin{align}
        \label{eq: refined estimate 1}
        \bigNorm{\tF_k^{D(\Aell^m)}(\hk D_t)   u}{H_{D(\Aell^m)}}^2
        \leq C \hk^{2(\ell-2m)}\E_\ell(u).
      \end{align}
\end{lemma}
The definition of the Fourier multiplier $\tF_k^{D(\Aell^m)}(\hk D_t)$ is as in the beginning
of Section~\ref{sec: Time microlocalization 1} for $F_k^{D(\Aell^m)}(\hk D_t)$.
Combined with the  third item of Lemma~\ref{lemma: semicl reduc action
  on L2loc} one has the following corollary.
\begin{corollary}
  \label{cor: refined estimate 1}
  If $\varphi \in \Cinfc(]0,\T[)$ and $\psi \in L^\infty(\R)$
     is such that $\psi =1$ in a \nhd of $\ovl{I_0}$, then for any $M\geq 1$ and
      $\ell \in \R$ there exists $C= C_{M,m, \ell}>0$ such that 
      \begin{align}
        \label{eq: cor refined estimate 1}
        \bigNorm{\varphi F_k^{D(\Aell^m)}(\hk D_t) 
        (1- \psi) \tF_k^{D(\Aell^m)}(\hk D_t)   u}{L^2(\R; D(\Aell^m))}^2
        \leq C \hk^M \E_\ell(u).
      \end{align}
\end{corollary}

\medskip
\begin{proof}[Proof of Lemma~\ref{lemma: refined estimate 1}]
  We consider the case $k >0$. The case $k<0$ is treated similarly.
  With Lemma~\ref{lemma: semicl reduc action on y} one has 
  \begin{align*}
    \tilde{u}^k (t)
    &=\sum_{\nu \in J_k}
      \tF_k \big( \hk \sqrt{\lambda_\nu}\big)
    e^{i t \sqrt{\lambda_\nu}} u_\nu^+ e_\nu \in E_k \subset D(\Aell^\infty).
  \end{align*}
  By Lemma~\ref{lemma: norm dt uk} one has 
  \begin{align*}
    \hk^{4m} \Norm{\tilde{u}^k (t)}{D(\Aell^m)}^2 
    &\eqsim \hk^{2\ell}\Norm{\tilde{u}^k (t)}{D(\Aell^{\ell/2})}^2 \\
    &\eqsim \hk^{2\ell} \sum_{\nu \in \tJ_k} \lambda_\nu^\ell       
      F_k \big( \hk\sqrt{\lambda_\nu}\big)^2
       \big| u_\nu^+\big|^2\\
    & \lesssim \hk^{2\ell} \sum_{\nu \in \tJ_k} \lambda_\nu^\ell       \big| u_\nu^+\big|^2,
  \end{align*}
  using that $F_k$  is a bounded function since
  compactly supported.
 This gives 
  \begin{align*}
    \Norm{\tilde{u}^k (t)}{D(\Aell^m)}^2 
    \lesssim \hk^{2(\ell-2m)}  \E_\ell(u).
  \end{align*}
  The result follows from the definition of $\Norm{.}{H_{D(\Aell^m)}}$
  in \eqref{eq: def H - abstract arg}.
\end{proof}

\medskip
A second important result is given by the following lemma. 
\begin{lemma}
  \label{lemma: second term time microlocalization}
  Let $\varphi \in \Cinfc(]0,\T[)$ and $\psi \in \Cinfc(\R)$
     be such that $\psi =1$ in a \nhd of $\ovl{I_0}$, then for any $M\geq 1$ and
      $\ell \in \R$ there exists $C= C_{M,m, \ell}>0$ such that 
      \begin{align}
        \label{eq: second term- refined}
        \bigNorm{\varphi F_k^{D(\Aell^m)}(\hk D_t) 
        (1- \psi) \big(\Id -\tF_k^{D(\Aell^m)}(\hk D_t)\big)   u}{L^2(\R; D(\Aell^m))}^2
        \leq C \hk^M \E_\ell(u).
      \end{align}
\end{lemma}
Note that here one assumes the function $\psi$ to be smooth as opposed
to the results in Lemma~\ref{lemma: semicl reduc action on L2loc} and
Corollary~\ref{cor: refined estimate 1}. In fact, the proof of
Lemma~\ref{lemma: second term time microlocalization} is based on a
kernel regularization argument that requires smoothness of the
function $\psi$.

\begin{proof}
As in other proofs we treat the case $k>0$. The case $k<0$ can be
treated similarly. 

With Lemma~\ref{lemma: semicl reduc action on L2loc} the proof is
clear in the case $\ell \geq 2m$. We shall thus only consider the case $\ell
< 2m$. Let $r \in \N$ be such that $r \geq m -\ell/2>0$. 

With $u$ solution to the wave equation  one has $u = -
\d_t^2\Aell^{-1} u = D_t^2\Aell^{-1} u$
and thus $u = D_t^{2 r}\Aell^{-r} u$. Set $w = \Aell^{-r} u \in
\cap_k \Con^k \big(\R; D(\Aell^{m+r-k})\big)$. It is also a solution
to  the wave equation. One has 
\begin{align}
  \label{eq: climbing energy scale}
  \E_{2m}(w) = \E_{2m-2r}(u) \lesssim \E_{\ell}(u).
\end{align} 
One thus considers the action of the operator 
\begin{align*}
  P = \varphi F_k^{D(\Aell^m)}(\hk D_t) 
        (1- \psi) \big(\Id -\tF_k^{D(\Aell^m)}(\hk D_t)\big) D_t^{2r},
\end{align*}
on $w$. Note that $P$ maps $\S'\big(\R; D(\Aell^m)\big)$ into
itself. Thus the action of $P$ on $\sum_{\indj\in \ZZ} w_\indj$ yields
$\sum_{\indj\in \ZZ} P w_\indj$ with convergence in $\S'\big(\R; D(\Aell^m)\big)$.
Recall that $I_\indj =  [\indj \T, (\indj+1)\T[$ and $w_\indj =
1_{I_\indj} w$.

The kernel of this operator is given by the following oscillatory
integral
\begin{align*}
  K(t,s) = (2 \pi)^{-2} \varphi(t) \iint \!\!\!\int
  e^{i(t-t')\tau' +i(t'-s)\tau} 
  \big(1-\psi (t')\big)   F_k(\hk \tau') \big(1-\tF_k(\hk \tau)\big) \tau^{2r} 
  \, d\tau' \,  d t' \,   d \tau .
\end{align*}
References on the  subject of oscillatory integrals are
\cite{Hoermander:V1,AG:91,LRLR:V1}. In particular, usual operations
such as integrations by parts are licit.

Since $\supp(F_k) \cap \supp ( 1 - \tF_k) = \emptyset$ one has $\tau'
\neq \tau$ in the integrand. In fact one has the following estimation.

\begin{lemma}
  \label{lemma: estimate diff tau tau'}
  There exists $C>0$ such that for all $k \in \ZZ^*$ one has 
  \begin{align}
    \label{eq: estimate diff tau tau'}
    |\tau - \tau'|^{-1} \leq C \min (\hk, \tau^{-1}),
\end{align}
if $\hk \tau' \in \supp(F_k)$ and $\hk \tau \in \supp(1-\tF_k)$.
\end{lemma}
A proof of Lemma~\ref{lemma: estimate diff tau tau'} is given below.

\medskip
With $\frac{-i }{\tau - \tau'}\d_{t'} e^{i t'(\tau - \tau')} = e^{i
  t'(\tau - \tau')}$, $N$ integrations by parts give
\begin{align*}
  K(t,s) = - i^N (2 \pi)^{-2} \varphi(t) \iint \!\!\!\int
  e^{i(t-t')\tau' +i(t'-s)\tau} 
  \psi^{(N)} (t')
  \frac{\tau^{2r} F_k(\hk \tau') \big(1-\tF_k(\hk \tau)\big)}{(\tau - \tau')^N} 
  \, d\tau' \,  d t' \,   d \tau .
\end{align*}
This is the step of the proof where smoothness of the function $\psi$
is used. 

Observe that $t \neq t'$ if $t \in \supp(\varphi)$ and $t' \in
\supp(1- \psi)$ or $\supp(\psi^{(N)})$.  
With  $\frac{-i }{t - t'}\d_{\tau'} e^{i \tau'(t - t')} = e^{i \tau'(t
  - t')}$, $N'$  integrations by parts give
\begin{multline*}
  K(t,s) = - i^{N+N'} \hk^{N'} (2 \pi)^{-2} \varphi(t) \iint \!\!\!\int
  e^{i(t-t')\tau' +i(t'-s)\tau} 
  \psi^{(N)} (t')\\
  \times \frac{\tau^{2r} F_k^{(N')}(\hk \tau') \big(1-\tF_k(\hk \tau)\big)}{(\tau -
  \tau')^N (t-t')^{N'}} 
  \, d\tau' \,  d t' \,   d \tau .
\end{multline*}
Using $N \geq 2r+2$ and Lemma~\ref{lemma: estimate diff tau tau'}
with this form of the kernel of $P$ a first estimate one can write is
the following
\begin{align}
  \label{eq: second term- refined-1}
  \bigNorm{P w_\indj(t)}{D(\Aell^m)}
  \lesssim \T^{1/2} \hk^{N'-1} \Norm{\varphi}{L^\infty} \Norm{\psi^{(N)}}{L^1} \Norm{F_k^{(N')}}{L^1} 
  \Norm{1-\tF_k}{L^\infty}
  \Norm{w_\indj}{L^2(\R; D(\Aell^m))}, \quad t \in \R.
\end{align}

\medskip
For $\indj$ such that $\dist (\supp(\psi), I_\indj)>0$,  we can proceed as in the proof of
Lemma~\ref{lemma: semicl reduc action on L2loc}. 
Set $G_k (\sigma) = \sigma^{2r} (1-\tF_k) (\sigma)$. One has 
\begin{multline*}
  K(t,s) = - i^{N+N'} \hk^{N'-2r} (2 \pi)^{-2} \varphi(t) \iint \!\!\!\int
  e^{i(t-t')\tau' +i(t'-s)\tau} 
  \psi^{(N)} (t')\\
  \times \frac{ F_k^{(N')}(\hk \tau') G_k(\hk \tau)}{(\tau -
  \tau')^N (t-t')^{N'}} 
  \, d\tau' \,  d t' \,   d \tau .
\end{multline*}

Set $\gamma =
\dist(\supp(\psi), I_\indj)$. One has $\frac{- i}{t'-s} \d_\tau
e^{i(t'-s)\tau} = e^{i(t-t')\tau' +i(t'-s)\tau}$. Thus, $N''$
integration by parts yield
\begin{multline*}
   P w_\indj(t) =  i^{N+N'+N''} \hk^{N'+N''-2 r} (2 \pi)^{-2} \varphi(t) \iint\!\!\!\iint
  e^{i(t-t')\tau' +i(t'-s)\tau} 
  \psi^{(N)} (t')\\
  \times \frac{F_k^{(N')}(\hk \tau') G_k^{(N'')}(\hk \tau)}{(\tau -
  \tau')^N (t-t')^{N'} (t'-s)^{N''}} w_\indj(s)
  \, d\tau' \,  d t' \,   d \tau d s.
\end{multline*}
If $N'' \geq 2r +1$ then $\supp(G_k^{(N'')})\subset \supp(\tF_k')
\subset \supp(\tF_k)$. 
Using Lemma~\ref{lemma: estimate diff tau tau'} one
obtains the following estimate
\begin{align}
  \label{eq: second term- refined-2}
  &\bigNorm{P w_\indj(t)}{D(\Aell^m)} \\
  &\quad \lesssim \T^{1/2} \hk^{N'+N''-2r-2} \gamma^{-N''}
   \Norm{\varphi}{L^\infty} \Norm{\psi^{(N)}}{L^1} 
  \Norm{F_k^{(N')}}{L^1} 
  \Norm{G_k^{(N'')}}{L^1}
  \Norm{w_\indj}{L^2(\R; D(\Aell^m))}, \quad t \in \R.\notag
\end{align}
If one chooses $N'' \geq 2$, 
with the $\gamma^{-N''}$ factor the sum with respect to $\indj$
converges. Since $\sum_\indj w_\indj$ converges to $w$ in $\S'\big(\R; D(\Aell^m)\big)$ one
concludes that  the action of $P$ on $w$ is equal to 
$\sum_\indj P w_\indj$ in $\S'\big(\R; D(\Aell^m)\big)$ and thus in
$L_{\loc}^\infty\big(\R; D(\Aell^m)\big)$ by estimate
\eqref{eq: second term- refined-2} for $|\indj|$ \suff large and  estimate \eqref{eq: second
  term- refined-1} for the remaining  finite number of
terms. Moreover, one has 
\begin{align*}
  \bigNorm{P w(t)}{D(\Aell^m)}
  \lesssim C_M \hk^M \sup_{\indj \in \ZZ}\Norm{w_\indj}{L^2(\R; D(\Aell^m))} 
  = C_M \hk^M \Norm{w}{H_{D(\Aell^m)}}\lesssim C_M \hk^M \E_{2m} (w)^{1/2},
\end{align*}
for any $M \in \N$.
As 
$\varphi F_k^{D(\Aell^m)}(\hk D_t) 
        (1- \psi) \big(\Id -\tF_k^{D(\Aell^m)}(\hk D_t)\big) u (t) 
  = P w(t)$, one concludes the proof with \eqref{eq: climbing energy scale}.
\end{proof}

\medskip
\begin{proof}[Proof of Lemma~\ref{lemma: estimate diff tau tau'}]
  If $\hk \tau' \in \supp(F_k)$ then $\alpha \leq \hk \tau'\leq
  \alpha^{-1}$. If $\hk \tau \in \supp(1-\tF_k)$ then $\hk|\tau -
  \tau'| \gtrsim 1$ yielding 
  \begin{align}
     \label{eq: estimate diff tau tau' -1}
    |\tau - \tau'|^{-1} \lesssim \hk.
  \end{align}

  First,  consider the case $\hk \tau \leq 2 \alpha^{-1}$. Then, $\hk
  \lesssim \tau^{-1}$. With \eqref{eq: estimate diff tau tau' -1}
  one obtains the result. 

  \medskip
  Second,  consider the case $\hk \tau \geq 2 \alpha^{-1}$. 
  Then, one has 
  \begin{align*}
    \hk |\tau - \tau'| = \hk \tau - \hk \tau' 
    \geq  \frac12 \hk \tau  
    + \big( \frac12 \hk \tau   - \alpha^{-1}\big)\geq \frac12 \hk \tau,
  \end{align*} 
  implying $|\tau - \tau'|^{-1} \lesssim \tau^{-1}$, yielding the
  result in this second case.
\end{proof}

\subsection{General case: $\ell_1 \leq 2 m_0$.}
\label{sec: proof ell1 < ell0}
The assumed
semi-classical observation inequality \eqref{eq: semicl obs} reads 
\begin{equation}
  \label{eq: semicl obs-reformulated-2}
  \E_{\ell_1}(u^k)
  \leq C \int _\delta^{\T-\delta} \Norm{\Obs u^k(t)}{K}^2dt,
  \qquad (u^k)_{k\in \N} \in B^+, \ k \geq k_0,
\end{equation}
and we aim to prove that 
\begin{equation}
   \label{eq: observation wave-2}
\E_{\ell_1}(u)
  \leq C' \int_{-\delta}^{\T+\delta} \Norm{\Obs u(t)}{K}^2\, d t, 
\end{equation}
holds for a solution $u$ to the wave equation \eqref{eq: abstract wave
  equation} writen in \eqref{eq: semicl reduc wave eq solution} that
lies in the $(2m_0)$-energy level.
Thus, we consider $\udl{u}^0\in D(\Aell^{m_0})$ and $\udl{u}^1  \in
D(\Aell^{m_0 -1/2})$. Then, $u \in \cap_k \Con^k \big(\R; D(\Aell^{m_0 - k/2})\big)$. 

The begining of the proof is similar to that given in
Section~\ref{sec: proof ell1 = ell0} and  
one reaches the following estimate that is the counterpart to \eqref{eq: semicl reduc  estimation energy} 
\begin{align}
  \label{eq: semicl reduc  estimation energy-2}
  \E_{\ell_1}(u)
  &\lesssim  \sum_{0 \leq  |k| < k_1}  \E_{\ell_1} (u^k)
  +  \sum_{|k|\geq k_1}\Norm{\varphi F_k^{K}(\hk D_t) \Obs u}{L^2(\R;K)}^2,
\end{align}
for $k_1 \geq k_0$ to be chosen below. The treatment of the terms in
the second sum is different from what is done in  Section~\ref{sec:
  proof ell1 = ell0}.
Consider $\psi \in \Cinfc(]-\delta, \T+\delta[)$ such that $\psi = 1$
in a \nhd of $\ovl{I_0} =  [0,\T]$.  One writes 
\begin{align*}
  \Norm{\varphi F_k^{K}(\hk D_t) \Obs u}{L^2(\R;K)}^2
  \lesssim 
  \Norm{\varphi F_k^{K}(\hk D_t) \psi \Obs u}{L^2(\R;K)}^2
  + \Norm{\varphi F_k^{K}(\hk D_t) (1-\psi)\Obs u}{L^2(\R;K)}^2,
\end{align*}
yielding, with the second  point of lemma~\ref{lemma: semicl reduc action on
  L2loc} 
\begin{align*}
  \E_{\ell_1}(u)
  &\lesssim  \sum_{0 \leq  |k| < k_1}  \E_{\ell_1} (u^k)
    +  \Norm{\psi \Obs u}{L^2(\R;K)}^2
   + \sum_{|k|\geq k_1} 
      \Norm{\varphi F_k^{K}(\hk D_t) (1-\psi)\Obs u}{L^2(\R;K)}^2.
      \notag
\end{align*}
We now concentrate our attention on the terms in the last sum on the
\rhs. First one writes
\begin{align*}
  \Norm{\varphi F_k^{K}(\hk D_t) (1-\psi)\Obs u}{L^2(\R;K)}
  \lesssim \Norm{\varphi F_k^{D(\Aell^{m_0})}(\hk D_t) (1-\psi) u}{L^2(\R;D(\Aell^{m_0}))},
\end{align*}
using that $\Obs$ is bounded on $D(\Aell^{m_0})$; see \eqref{eq: bound
  obs op}. This gives
\begin{align}
  \label{eq: semicl reduc  estimation energy-2b}
  \E_{\ell_1}(u)
  &\lesssim  \sum_{0 \leq  |k| < k_1}  \E_{\ell_1} (u^k)
    +  \Norm{\psi \Obs u}{L^2(\R;K)}^2\\
    &\quad + \sum_{|k|\geq k_1} 
      \Norm{\varphi F_k^{D(\Aell^{m_0})}(\hk D_t) (1-\psi) u}{L^2(\R;D(\Aell^{m_0}))}^2.
      \notag
\end{align}

Second, as in Section~\ref{sec: Refined time-microlocalization estimates}
consider $\tF \in \Cinfc(\R_+^*)$ such that $\tF=1$ in a \nhd of
$\supp(F)$. With Corollary~\ref{cor: refined estimate 1} and Lemma~\ref{lemma: second term time microlocalization} one has 
\begin{align*}
  &\Norm{\varphi F_k^{D(\Aell^{m_0})}(\hk D_t) (1-\psi) u}
  {L^2(\R;D(\Aell^{m_0}))}^2\\
  &\quad \lesssim 
  \bigNorm{\varphi F_k^{D(\Aell^{m_0})}(\hk D_t) 
        (1- \psi) \tF_k^{D(\Aell^{m_0})}(\hk D_t)   u}
  {L^2(\R; D(\Aell^{m_0})}^2\\
 &\qquad  + \bigNorm{\varphi F_k^{D(\Aell^{m_0})}(\hk D_t) 
        (1- \psi) \big(\Id -\tF_k^{D(\Aell^{m_0})}(\hk D_t)\big)   u}{L^2(\R; D(\Aell^{m_0}))}^2\\
  &\quad \lesssim \hk^M \E_{\ell_1}(u),
\end{align*}
for any $M \in \N$. From~\eqref{eq: semicl reduc  estimation
  energy-2b} using that  $\hk = \rho^{-|k|}$ with $\rho>1$
one obtains
\begin{align*}
  \E_{\ell_1}(u)
  &\lesssim  \sum_{0 \leq  |k| < k_1}  \E_{\ell_1} (u^k)
    +  \Norm{\psi \Obs u}{L^2(\R;K)}^2
   + h_{k_1}^M \E_{\ell_1}(u).
\end{align*}
For $k_1\geq k_0$ chosen \suff large one obtains
\begin{align}
  \label{eq: semicl reduc  estimation energy2-2}
  \E_{\ell_1}(u)
  &\lesssim  \sum_{0 \leq  |k| < k_1}  \E_{\ell_1}(u^k)
  +  \Norm{\psi \Obs u}{L^2(\R; K)}^2.
\end{align}
With \eqref{eq: semicl reduc  estimation energy2-2} the result of
Lemma~\ref{lemma: invisible solution} holds here too.
Arguing as in the proof given in Section~\ref{sec: proof ell1 = ell0}
one obtains the sought observability estimate.
\hfill \qedsymbol \endproof

\section{Proof of the main result for the Schr\"odinger equation}

\subsection{Refined time-microlocalization estimates}
\label{sec: Refined time-microlocalization estimates Schrodinger}

The results of Section~\ref{sec: Refined time-microlocalization
  estimates} can be adapted to a solution $u(t)$ of a Schr\"odinger
equation \eqref{eq: abstract Schrodinger equation} with
$\udl{u}^0 \in D(\Aell^m)$, for some $m \in \R$. Then,
$u \in \cap_k \Con^k\big(\R; D(\Aell^{m-k})\big)$.

Let $F \in \Cinfc(]\alpha, \alpha^{-1}[)$ be as given by Lemma~\ref{lemma: semicl
  reduc choice F}. Consider $\tF \in \Cinfc(]\alpha, \alpha^{-1}[)$ such that $\tF=1$
in a \nhd of $\supp(F)$. 
\begin{lemma}
  \label{lemma: refined estimate 1-Schrodinger}
  Let $p \in \R$. There exists $C= C_{m, p}>0$ such that  
      \begin{align}
        \label{eq: refined estimate 1-Schrodinger}
        \bigNorm{\tF_k^{D(\Aell^m)}(\hk D_t)   u}{H_{D(\Aell^m)}}
        \leq C \hk^{p-m} \Norm{\udl{u}^0}{D(\Aell^p)}, \qquad k \in \N^*.
      \end{align}
\end{lemma}
The definition of the Fourier multiplier $\tF_k^{D(\Aell^m)}(\hk D_t)$ is as in the beginning
of Section~\ref{sec: Time microlocalization 1}.
\begin{proof}
  With Lemma~\ref{lemma: semicl reduc action on y Schrodinger} one has
   \begin{align*}
     \tilde{u}^k(t) = \tF_k^{D(\Aell^{m})}(\hk D_t) u (t)
    =\sum_{\nu \in J^S_k}
      \tF_k \big( \hk \lambda_\nu\big)
    e^{i t \lambda_\nu} \udl{u}^0_\nu e_\nu \in E^S_k \in D(\Aell^\infty).
\end{align*} 
With Lemma~\ref{lemma: norm dt uk Schrodinger} one writes
  \begin{align*}
    \hk^{m} \Norm{\tilde{u}^k(t)}{D(\Aell^m)} 
    &\eqsim \hk^{p}\Norm{\tilde{u}^k(t)}{D(\Aell^{p})} 
      \lesssim \hk^{p} \Norm{ u (t)}{D(\Aell^{p})}
      \lesssim \hk^{p} \Norm{\udl{u}^0}{D(\Aell^p)}.
  \end{align*}
  The result follows from the definition of $\Norm{.}{H_{D(\Aell^m)}}$
  in \eqref{eq: def H - abstract arg}.
\end{proof}
Combined with the  third item of Lemma~\ref{lemma: semicl reduc action
  on L2loc} one has the following corollary.
\begin{corollary}
  \label{cor: refined estimate 1-Schrodinger}
  If $\varphi \in \Cinfc(]0,\T[)$ and $\psi \in L^\infty(\R)$
     is such that $\psi =1$ in a \nhd of $\ovl{I_0}$, then for any $M\geq 1$ and
      $p \in \R$ there exists $C= C_{M,m, p}>0$ such that 
      \begin{align}
        \label{eq: cor refined estimate 1-Schrodinger}
        \bigNorm{\varphi F_k^{D(\Aell^m)}(\hk D_t) 
        (1- \psi) \tF_k^{D(\Aell^m)}(\hk D_t)   u}{L^2(\R; D(\Aell^m))}^2
        \leq C \hk^M \Norm{\udl{u}^0}{D(\Aell^p)}.
      \end{align}
\end{corollary}

\medskip
\begin{lemma}
  \label{lemma: second term time microlocalization Schrodinger}
  Let $\varphi \in \Cinfc(]0,\T[)$ and $\psi \in \Cinfc(\R)$
     be such that $\psi =1$ in a \nhd of $\ovl{I_0}$, then for any $M\geq 1$ and
      $p \in \R$ there exists $C= C_{M,m, p}>0$ such that 
      \begin{align}
        \label{eq: second term- refined Schrodinger}
        \bigNorm{\varphi F_k^{D(\Aell^m)}(\hk D_t) 
        (1- \psi) \big(\Id -\tF_k^{D(\Aell^m)}(\hk D_t)\big)   u}{L^2(\R; D(\Aell^m))}^2
        \leq C \hk^M \Norm{\udl{u}^0}{D(\Aell^p)}.
      \end{align}
\end{lemma}
\begin{proof}
With Lemma~\ref{lemma: semicl reduc action on L2loc} the proof is
clear in the case $p \geq m$. We shall thus only consider the case $p
< m$. Let $r \in \N$ be such that $r \geq m -p$.  With $u$ solution to the Schrodinger equation  one has $u = D_t^{r}\Aell^{-r} u$. Set $w = \Aell^{-r} u$. It is also a solution of the Schrodinger equation that lies in
$D(\Aell^\infty)$. One has 
\begin{align}
  \label{eq: climbing energy scale Schrodinger}
  \Norm{w(t)}{D(\Aell^m)}= \Norm{u(t)}{D(\Aell^{m-r})} 
  \lesssim \Norm{u(t)}{D(\Aell^{p})}, 
  \quad t \in \R.
\end{align} 
One thus considers the action of the operator 
\begin{align*}
  P = \varphi F_k^{D(\Aell^m)}(\hk D_t) 
        (1- \psi) \big(\Id -\tF_k^{D(\Aell^m)}(\hk D_t)\big) D_t^{r},
\end{align*}
on $w$. As in the proof of Lemma~\ref{lemma: second term time
  microlocalization} one obtains
\begin{align*}
  \bigNorm{P w(t)}{D(\Aell^m)}
  \lesssim C_M \hk^M \sup_{\indj \in \ZZ}\Norm{w_\indj}{L^2(\R; D(\Aell^m))} 
  = C_M \hk^M \Norm{w}{H_{D(\Aell^m)}} \lesssim   C_M \hk^M \Norm{w(0)}{D(\Aell^m)}.
\end{align*}
As 
$\varphi F_k^{D(\Aell^m)}(\hk D_t) 
        (1- \psi) \big(\Id -\tF_k^{D(\Aell^m)}(\hk D_t)\big) u (t) 
  = P w(t)$, one concludes the proof with \eqref{eq: climbing energy scale Schrodinger}.
\end{proof}

\subsection{Proof of Theorem~\ref{theorem: main result Schrodinger}}
Here, we provide only one proof that treats the general case $p_1
\leq m_0$. Let $\T = (T+T')/2$ and $\delta = (\T-T)/2 =
(T'-\T)/2$. The assumed
semi-classical observation inequality reads 
\begin{equation}
  \label{eq: semicl obs-Schrodinger proof} 
  \Norm{u^k}{D(\Aell^{p_1}}
  \leq C \int _\delta^{\T-\delta} \Norm{\Obs u^k(t)}{K}dt,
  \qquad (u^k)_{k\in \N} \in B^S, \ k \geq k_0,
\end{equation}
and we aim to prove that 
\begin{equation*}
 \Norm{\udl{u}^0}{D(\Aell^{p_1})}
  \leq C' \int_{-\delta}^{\T+\delta} \Norm{\Obs u(t)}{K}\, d t 
\end{equation*}
holds for any solution $u$ to the Schr\"odinger equation \eqref{eq: abstract Schrodinger
  equation} writen in \eqref{eq: semicl reduc Schrodinger eq solution}
with $\udl{u}^0\in D(\Aell^{m_0})$. We thus consider such a
solution. One has $u \in \Con^k\big(\R; D(\Aell^{m_0-k})\big)$.

One has 
\begin{align}
  \label{eq: observation in H-schrodinger}
  \Norm{\Obs u}{H_K} 
  \lesssim \Norm{\Obs u}{L^\infty(\R; K)} 
  \lesssim \Norm{u}{L^\infty(\R; D(\Aell^{m_0})}  
  \lesssim \Norm{\udl{u}^0}{D(\Aell^{m_0})} .
\end{align}
Let $F$ be chosen as in Lemma~\ref{lemma: semicl reduc choice
  F}. With Lemma~\ref{lemma: semicl reduc
  action on y Schrodinger}, for $k \in \N^*$, set
\begin{align}
  \label{eq: chi k y Schrodinger}
  u^k (t)= F_k^{D(\Aell^{p_1})}(\hk D_t) u(t)
  = \sum_{\nu \in J_k^S}
   F_k ( \hk \lambda_\nu)
    e^{i t \lambda_\nu} \udl{u}^0_\nu e_\nu.
\end{align}
One has $u^k \in E_k$.
For $k \geq k_0$ with 
the semi-classical observation property~\eqref{eq: semicl obs-Schrodinger proof} one has
\begin{equation}
  \label{eq: semicl obs-bis Schrodinger}
  \Norm{u^k}{D(\Aell^{p_1})} 
  \lesssim \Norm{\varphi \Obs u^k(t)}{L^2(\R,K)}, 
  \qquad |k| \geq k_0,
\end{equation}
where   $\varphi \in \Cinfc(]0,\T[)$ with $\varphi =1$ on a \nhd of 
$[\delta,\T-\delta]$.
In the cases $k>0$
or $k <0$, one has 
\begin{align*}
  \Norm{u^k}{D(\Aell^{p_1})}^2
  = \sum_{\nu \in J_k^S}\lambda_\nu^{2 p_1}
  F_k(\hk \lambda_\nu)^2 |\udl{u}^0_\nu|^2.
\end{align*}
Set $u^0 = \sum_{\lambda_\nu \leq 1} e^{i t \lambda_\nu}
\udl{u}^0_\nu$.
With Lemma~\ref{lemma: semicl reduc choice F} one finds
\begin{align*}
  \Norm{u- u^0}{D(\Aell^{p_1})}^2
  &=  \Norm{u}{D(\Aell^{p_1})}^2 - \Norm{u^0}{D(\Aell^{p_1})}^2
    =\sum_{\lambda_\nu>1} \lambda_\nu^{2 p_1}
     |\udl{u}^0_\nu|^2
  \lesssim
    \sum_{k \in \N^*} \sum_{\lambda_\nu>1}
    \lambda_\nu^{2p_1} F_k(\hk \lambda_\nu)^2 |\udl{u}^0_\nu|^2\\
  & \lesssim \sum_{k \in \N^*}  \Norm{u^k}{D(\Aell^{p_1})}^2.
\end{align*}
One thus obtains with \eqref{eq: semicl obs-bis Schrodinger}
\begin{align}
  \label{eq: separation energy Schrodinger}
  \Norm{u}{D(\Aell^{p_1})}^2&\lesssim  \Norm{u^0}{D(\Aell^{p_1})}^2 + \sum_{k \in \N^*}  \Norm{u^k}{D(\Aell^{p_1})}^2\\
  &\lesssim  \Norm{u^0}{D(\Aell^{p_1})}^2 + \sum_{1 \leq |k| < k_1}   \Norm{u^k}{D(\Aell^{p_1})}^2
    + \sum_{|k|\geq k_1}\Norm{\varphi \Obs u^k(t)}{L^2(\R;K)}^2,
    \notag
\end{align}
for $k_1 \geq k_0$ to be chosen below.

With  Lemma~\ref{lemma: semicl filter Schrodinger} one has
\begin{align}
  \label{eq: semicl reduc  estimation energy Schrodinger}
   \Norm{u}{D(\Aell^{p_1})}^2
  &\lesssim  \sum_{0 \leq  |k| < k_1}  \Norm{u^k}{D(\Aell^{p_1})}^2
  +  \sum_{|k|\geq k_1}\Norm{\varphi F_k^{K}(\hk D_t) \Obs u}{L^2(\R;K)}^2.
\end{align}
Consider $\psi \in \Cinfc(]-\delta, \T+\delta[)$ such that $\psi = 1$
in a \nhd of $\ovl{I_0} =  [0,\T]$.  One writes 
\begin{align*}
  \Norm{\varphi F_k^{K}(\hk D_t) \Obs u}{L^2(\R;K)}^2
  \lesssim 
  \Norm{\varphi F_k^{K}(\hk D_t) \psi \Obs u}{L^2(\R;K)}^2
  + \Norm{\varphi F_k^{K}(\hk D_t) (1-\psi)\Obs u}{L^2(\R;K)}^2,
\end{align*}
yielding, with the second  point of lemma~\ref{lemma: semicl reduc action on
  L2loc} 
\begin{align*}
 \Norm{u}{D(\Aell^{p_1})}^2
  &\lesssim  \sum_{0 \leq  |k| < k_1}  \Norm{u^k}{D(\Aell^{p_1})}^2
    +  \Norm{\psi \Obs u}{L^2(\R;K)}^2
   + \sum_{|k|\geq k_1} 
      \Norm{\varphi F_k^{K}(\hk D_t) (1-\psi)\Obs u}{L^2(\R;K)}^2.
      \notag
\end{align*}
One writes
\begin{align*}
  \Norm{\varphi F_k^{K}(\hk D_t) (1-\psi)\Obs u}{L^2(\R;K)}
  \lesssim \Norm{\varphi F_k^{D(\Aell^{m_0})}(\hk D_t) (1-\psi) u}{L^2(\R;D(\Aell^{m_0}))},
\end{align*}
using that $\Obs$ is bounded on $D(\Aell^{m_0})$; see \eqref{eq: bound
  obs op}. This gives
\begin{align}
  \label{eq: semicl reduc  estimation energy-2b-Schrodinger}
  \Norm{u}{D(\Aell^{p_1})}^2
  &\lesssim  \sum_{0 \leq  |k| < k_1}   \Norm{u^k}{D(\Aell^{p_1})}^2
    +  \Norm{\psi \Obs u}{L^2(\R;K)}^2\\
    &\quad + \sum_{|k|\geq k_1} 
      \Norm{\varphi F_k^{D(\Aell^{m_0})}(\hk D_t) (1-\psi) u}{L^2(\R;D(\Aell^{m_0}))}^2.
      \notag
\end{align}

Second, as in Section~\ref{sec: Refined time-microlocalization estimates}
consider $\tF \in \Cinfc(\R_+^*)$ such that $\tF=1$ in a \nhd of
$\supp(F)$. With Corollary~\ref{cor: refined estimate 1-Schrodinger} and Lemma~\ref{lemma: second term time microlocalization Schrodinger} one has 
\begin{align*}
  &\Norm{\varphi F_k^{D(\Aell^{m_0})}(\hk D_t) (1-\psi) u}
  {L^2(\R;D(\Aell^{m_0}))}^2\\
  &\quad \lesssim 
  \bigNorm{\varphi F_k^{D(\Aell^{m_0})}(\hk D_t) 
        (1- \psi) \tF_k^{D(\Aell^{m_0})}(\hk D_t)   u}
  {L^2(\R; D(\Aell^{m_0})}^2\\
 &\qquad  + \bigNorm{\varphi F_k^{D(\Aell^{m_0})}(\hk D_t) 
        (1- \psi) \big(\Id -\tF_k^{D(\Aell^{m_0})}(\hk D_t)\big)   u}{L^2(\R; D(\Aell^{m_0}))}^2\\
  &\quad \lesssim \hk^M  \Norm{u}{D(\Aell^{p_1})}^2,
\end{align*}
for any $M \in \N$. From~\eqref{eq: semicl reduc  estimation
  energy-2b-Schrodinger} using that  $\hk = \rho^{-|k|}$ with $\rho>1$
one obtains
\begin{align*}
  \Norm{u}{D(\Aell^{p_1})}^2
  &\lesssim  \sum_{0 \leq  |k| < k_1}  \Norm{u^k}{D(\Aell^{p_1})}^2
    +  \Norm{\psi \Obs u}{L^2(\R;K)}^2
   + h_{k_1}^M  \Norm{u}{D(\Aell^{p_1})}^2.
\end{align*}
For $k_1\geq k_0$ chosen \suff large one obtains
\begin{align}
  \label{eq: semicl reduc  estimation energy2-2-Schrodinger}
  \Norm{u}{D(\Aell^{p_1})}^2
  &\lesssim  \sum_{0 \leq  |k| < k_1}   \Norm{u^k}{D(\Aell^{p_1})}^2
  +  \Norm{\psi \Obs u}{L^2(\R; K)}^2.
\end{align}

The following lemma is the counterpart of Lemma~\ref{lemma: invisible solution}.
\begin{lemma}[absence of invisible solutions to the Schr\"odinger equation]
  \label{lemma: invisible solution Schrodinger}
  Let $u \in \cap_k \Con^k\big(\R; D(\Aell^{m_0-k})\big)$
be a solution to
  \eqref{eq: abstract Schrodinger equation} such that $\psi \Obs u =0$.
  Then $u=0$.
\end{lemma}
The proof is very similar to that of Lemma~\ref{lemma: invisible solution}.
\begin{proof}
  Set $\mathscr N_S$ as the space of such invisible solutions (in the sense of the
  observation operator $\psi\Obs$) equipped with the norm $\Norm{\udl{u}^0}{D(\Aell^{m_0})}$. With \eqref{eq: semicl reduc  estimation energy2-2-Schrodinger} one has
  $\Norm{\udl{u}^0}{D(\Aell^{m_0})}\lesssim  \sum_{0 \leq  |k| < k_1}  \Norm{u^k}{D(\Aell^{p_1})}^2$ implying $\mathscr
  N_S = \Span\{ e_\nu; \nu \in \Upsilon\}$ with $\# \Upsilon < \infty$. Moreover, if $u \in \mathscr N_S$ then $u
  \in \Con^m\big(\R, D(\Aell^r)\big)$ for any $m\geq0$ and $r\geq 0$,
  similarly to what one has in \eqref{eq: regularity yk}. On this finite dimensional space one has
  $\psi \Obs \d_t u = \psi \d_t \Obs u =0$. Thus $\d_t$ maps $\mathscr N_S$ into
  itself and consequently it has an eigenvector $\mathsf v$ with
  associated eigenvalue
  $\mu$. One finds $\Aell \mathsf v =  D_t  \mathsf v = - i \mu \mathsf
  v$ meaning that $\mathsf v(t)$ is an eigenfunction for $\Aell$ for all
  $t$. With the unique continuation Assumption~\ref{assumpt: unique continuation} one obtains
  $\mathsf v(t)=0$ for all $t$. Hence $\mathscr N_S =\{0\}$.
\end{proof}
\medskip
We conclude the proof of Theorem~\ref{theorem: main result Schrodinger}
by an argument by contradiction similar to that in  the proof of
Theorem~\ref{theorem: main result - wave}. Adaptation is left to the reader.
\hfill \qedsymbol \endproof


\begin{thebibliography}{10}

\bibitem{AG:91}
S.~Alinhac and P.~G\'erard.
\newblock {\em Op\'erateurs {P}seudo-{D}iff\'erentiels et {T}h\'eor\`eme de
  {N}ash-{M}oser}.
\newblock Editions du CNRS, 1991.

\bibitem{BLR:92}
C.~Bardos, G.~Lebeau, and J.~Rauch.
\newblock Sharp sufficient conditions for the observation, control, and
  stabilization of waves from the boundary.
\newblock {\em SIAM J. Control Optim.}, 30:1024--1065, 1992.

\bibitem{Burq-93}
N.~Burq.
\newblock Contr\^{o}le de l'\'{e}quation des plaques en pr\'{e}sence
  d'obstacles strictement convexes.
\newblock {\em M\'{e}m. Soc. Math. France (N.S.)}, 55:126, 1993.

\bibitem{Burq:1997}
N.~Burq.
\newblock Contr\^ole de l'\'equation des ondes dans des ouverts peu
  r\'eguliers.
\newblock {\em Asymptotic Analysis}, 14:157--191, 1997.

\bibitem{BDLR1}
N.~Burq, B.~Dehman, and J.~Le~Rousseau.
\newblock Measure and continuous vector field at a boundary {I}: propagation
  equation and wave observability.
\newblock {\em Submitted, 65 pages}, 2022.

\bibitem{Bu-Leb-92}
N.~Burq and G.~Lebeau.
\newblock Micro-local approach to the control for the plates equation.
\newblock In {\em Optimization, optimal control and partial differential
  equations ({I}a\c{s}i, 1992)}, volume 107 of {\em Internat. Ser. Numer.
  Math.}, pages 111--122. Birkh\"{a}user, Basel, 1992.

\bibitem{BuZw}
N.~Burq and M.~Zworski.
\newblock Geometric control in the presence of a black box.
\newblock {\em J. Amer. Math. Soc.}, 17(2):443--471, 2004.

\bibitem{Coron:2007}
J.-M. Coron.
\newblock {\em Control and nonlinearity}, volume 136 of {\em Mathematical
  Surveys and Monographs}.
\newblock American Mathematical Society, Providence, RI, 2007.

\bibitem{DL:09}
B.~Dehman and G.~Lebeau.
\newblock Analysis of the {HUM} control operator and exact controllability for
  semilinear waves in uniform time.
\newblock {\em SIAM J. Control Optim.}, 48:521--550, 2009.

\bibitem{Fan-Zua}
F.~Fanelli and E.~Zuazua.
\newblock Weak observability estimates for 1-{D} wave equations with rough
  coefficients.
\newblock {\em Ann. Inst. H. Poincar\'e Anal. Non Lin\'eaire}, 32(2):245--277,
  2015.

\bibitem{Hoermander:83}
L.~H\"ormander.
\newblock Uniqueness theorems for second order elliptic differential equations.
\newblock {\em Comm. Part. Diff. Equations}, 8(1):21--64, 1983.

\bibitem{Hoermander:V3}
L.~H\"ormander.
\newblock {\em The {A}nalysis of {L}inear {P}artial {D}ifferential
  {O}perators}, volume III.
\newblock Springer-Verlag, 1985.
\newblock Second printing 1994.

\bibitem{Hoermander:V1}
L.~H\"ormander.
\newblock {\em The {A}nalysis of {L}inear {P}artial {D}ifferential
  {O}perators}, volume~I.
\newblock Springer-Verlag, second edition, 1990.

\bibitem{LLT:86}
I.~Lasiecka, J.-L. Lions, and R.~Triggiani.
\newblock Non homogeneous boundary value problems for second order hyperbolic
  operators.
\newblock {\em J. Math. Pures Appl.}, 65:149--192, 1986.

\bibitem{LRLR:V1}
J.~Le~Rousseau, G.~Lebeau, and L.~Robbiano.
\newblock {\em {E}lliptic {C}arleman {E}stimates and {A}pplications to
  {S}tabilization and {C}ontrollability, Volume I: {D}irichlet {B}oundary
  {C}onditions on {E}uclidean {S}pace}.
\newblock PNLDE Subseries in Control. Birkh\"auser, 2022.

\bibitem{LRLR:V2}
J.~Le~Rousseau, G.~Lebeau, and L.~Robbiano.
\newblock {\em {E}lliptic {C}arleman {E}stimates and {A}pplications to
  {S}tabilization and {C}ontrollability, Volume II: {G}eneral {B}oundary
  {C}onditions on {R}iemnannian {M}anifolds}.
\newblock PNLDE Subseries in Control. Birkh\"auser, 2022.

\bibitem{Lebeau:92}
G.~Lebeau.
\newblock Contr\^ole de \'equation de {S}chr\"odinger.
\newblock {\em J. Math. Pures Appl.}, 71:267--291, 1992.

\bibitem{Lions:88}
J.-L. Lions.
\newblock {\em Contr\^olabilit\'e {E}xacte, {P}erturbations et {S}tabilisation
  de {S}yst\`emes {D}istribu\'es}, volume~1.
\newblock Masson, Paris, 1988.

\bibitem{Tu-1}
G.~Tenenbaum and M.~Tucsnak.
\newblock Fast and strongly localized observation for the schr\"odinger
  equation.
\newblock {\em Transactions of the AMS}, 361(2):951--977, 2009.

\bibitem{Treves:67}
F.~Treves.
\newblock {\em Topological {V}ector {S}paces, {D}istributions and {K}ernels}.
\newblock Academic Press, New York, 1967.

\bibitem{TW:09}
M.~Tucsnak and G.~Weiss.
\newblock {\em {O}bservation and {C}ontrol for {O}perator {S}emigroups}.
\newblock Birkh{\"a}user Verlag, Basel, 2009.

\end{thebibliography}

\end{document}